\documentclass[12pt]{smfart}

\usepackage[T1]{fontenc}
\usepackage{lmodern,amssymb,bm}

\usepackage[french, english]{babel}

\usepackage[utf8]{inputenc}

\usepackage{amsthm,amsmath,amssymb,amsfonts}
\usepackage{enumerate}
\usepackage{tikz}
\usetikzlibrary{arrows, matrix, shapes, patterns, decorations.pathmorphing}

\newcommand{\R}{\ensuremath{\mathbb{R}}}
\newcommand{\Z}{\ensuremath{\mathbb{Z}}}

\newcommand{\N}{\ensuremath{\mathbb{N}}}

\newcommand{\T}{\ensuremath{\mathbb{T}}}
\newcommand{\Sp}{\ensuremath{\mathbb{S}}}
\newcommand{\F}{\ensuremath{\mathcal{F}}}
\newcommand{\fix}{\ensuremath{\operatorname{Fix}}}
\newcommand{\dom}{\ensuremath{\operatorname{dom}}}
\newcommand{\inte}{\ensuremath{\operatorname{int}}}
\newcommand{\homeo}{\ensuremath{\operatorname{Homeo}}}
\newcommand{\id}{\ensuremath{\operatorname{id}}}
\newcommand{\wt}{\widetilde}

\usepackage[margin=3cm]{geometry}

\theoremstyle{plain}
\newtheorem{defi}{Définition}[section]
\newtheorem{prop}[defi]{Proposition}
\newtheorem{theo}[defi]{Théorème}

\newtheorem{coro}[defi]{Corollaire}

\theoremstyle{definition}
\newtheorem{lemm}[defi]{Lemme}
\theoremstyle{remark}

\theoremstyle{plain}

\title[Théorie de forçage des homéomorphismes de surface]{Théorie de forçage des homéomorphismes de surface\\
{[}d'après Le Calvez et Tal{]}}

\author[Pierre-Antoine Guih\'eneuf]{Pierre-Antoine Guih\'eneuf\\
Séminaire \bsc{Bourbaki}}

\date{Séminaire \bsc{Bourbaki}\\
$72^\text{e}$ année, 2019–2020, $\text{n}^\text{o}$ 1171\\
Janvier 2020}
\address{Sorbonne Universit{\'e}, Universit{\'e} Paris Diderot, CNRS, Institut de Math{\'e}matiques de Jussieu-Paris Rive Gauche, IMJ-PRG, F-75005, Paris, France}
\email{pierre-antoine.guiheneuf@imj-prg.fr}

\begin{document}

\begin{abstract}
En 1912 Brouwer publiait son théorème de translation, qui implique par exemple qu’un homéomorphisme du plan préservant l'orientation et ayant un point périodique possède aussi un point fixe. Ce théorème a donné lieu à bon nombre de développements, débouchant entres autres sur l'obtention par Le Calvez d'un feuilletage de Brouwer équivariant pour les homéomorphismes de surface homotopes à l'identité. Récemment, Le Calvez et Tal ont utilisé ce feuilletage pour construire une théorie de forçage par essence topologique qui, à l'instar du théorème de Brouwer, permet de déduire l’existence de nouvelles orbites à partir de certaines propriétés dynamiques de l'homéomorphisme. L'exposé décrira les principes généraux de cette théorie, ainsi que quelques unes de ses très nombreuses applications.
\end{abstract}

\begin{altabstract}
In 1912 Brouwer published his translation theorem, which implies, for example, that an orientation preserving homeomorphism of the plane having a periodic point also has a fixed point. This theorem has given rise to a number of developments, leading among other things
to Le Calvez's proof of the existence of a Brouwer foliation for surface homeomorphisms homotopic to identity. Recently, Le Calvez and Tal used this foliation to construct a forcing theory intrinsically topological which, like Brouwer's theorem, allows to deduce the existence of new orbits from certain dynamic properties of homeomorphism. The exposé will describe the general principles of this theory, as well as some of its many applications. 
\end{altabstract}


\maketitle

\tableofcontents



%
%
%
%
%
%
%
%
%
%
%
%
%

\section*{Introduction}

Cet exposé de dynamique commence de manière terriblement banale par une idée géniale de Poincaré. Quelques mois avant sa mort, s'excusant de n'avoir \og jamais présenté au public un travail aussi inachevé\fg, celui-ci se résout à publier des résultats partiels car, vu son âge, il n'est pas sûr de pouvoir les reprendre un jour. Son article est en grande partie motivé par l'étude des orbites périodiques du problème à trois corps dans le cas où, contrairement à ses travaux antérieurs, les masses ne sont pas petites. Utilisant une intégrale première et ce qui est désormais appelé une section de Poincaré, il ramène l'existence d'orbites périodiques à ce que certains nomment \og le dernier théorème de géométrie de Poincaré\fg.

\begin{theo}[Poincaré-Birkhoff]\label{ThPB}
Soit un homéomorphisme de l'anneau fermé préservant l'aire, et tel que les nombres de rotation de ses restrictions aux deux cercles du bord sont de signes opposés\footnote{Ce qui sous-entend que ceux-ci sont préservés, et que l'homéomorphisme préserve l'orientation.}. Alors cet homéomorphisme possède au moins deux points fixes.
\end{theo}

Poincaré explique qu'il sait tout de même résoudre un grand nombre de cas particuliers. \cite{MR1500933} démontrera rigoureusement ce théorème (pour l'existence d'un point fixe seulement) quelques mois après le décès de Poincaré (voir aussi \cite{IDM}) dans un tour de force qui devient l'un des tous premiers résultats difficiles en dynamique topologique.

Ce théorème aura une influence considérable à la fois pour la fondation de la topologie symplectique, avec la conjecture d'Arnold, et de la dynamique topologique sur les surfaces, dont il sera question dans cet exposé. 
\medskip

Une seconde motivation de la dynamique sur les surfaces est plus pragmatique : il s'agit de faire un compromis entre la réalité physique de systèmes ayant des espaces de configuration de très grande dimension et la puissance des outils mathématiques à notre disposition.
\medskip

Le but de cet exposé est de présenter les idées d'une théorie du forçage pour les homéomorphismes de surfaces, fondée dans \cite{LCT1} et \cite{LCT2}, et qui est en quelque sorte l'aboutissement de toute une série de travaux visant à améliorer la théorie de Brouwer. Si celle-ci prend sa source dans le 5\ieme{} problème de Hilbert, elle s'est rapidement imposée comme un outil incontournable pour la dynamique topologique sur les surfaces --- il s'est par ailleurs avéré que le théorème \ref{ThBrou} de translation de Brouwer donne une preuve du théorème de Poincaré-Birkhoff (voir \cite{MR1273787}).

L'idée de la théorie du forçage est d'exploiter le théorème \ref{TheoPatriceFeuillete} de \cite{PatriceFeuille2} qui associe à tout homéomorphisme de surface homotope à l'identité un feuilletage singulier sur la surface. On utilise alors une sorte de dualité entre l'action de l'homéomorphisme $f$ sur les points de la surface et sur les feuilles du feuilletage : comme résumé par le diagramme suivant, on part d'une propriété sur $f$, la traduit en termes de feuilletages, on en déduit l'existence de nouvelles orbites de l'action de $f$ sur le feuilletage, ce qui implique une certaine propriété de l'action de $f$ sur les points.

\begin{center}
\begin{tikzpicture}[scale=.8]
\node (A) at (0,0) {$S$};
\node (B) at (2,0) {$S$};
\node (C) at (0,2) {$\F$};
\node (D) at (2,2) {$\F'$};

\draw[->,>=latex,shorten >=3pt, shorten <=3pt] (A) to node[midway, above]{$f$} (B);
\draw[->,>=latex,shorten >=3pt, shorten <=3pt] (A) to (C);
\draw[->,>=latex,shorten >=3pt, shorten <=3pt] (C) to node[midway, above]{$f$} (D);
\draw[->,>=latex,shorten >=3pt, shorten <=3pt] (D) to (B);
\end{tikzpicture}
\end{center}

Le parti pris de cet exposé n'est pas de lister toutes les (innombrables) applications de cette théorie, mais plutôt d'aller calmement vers l'énoncé et la preuve de la proposition fondamentale \ref{PropFond} de forçage, avant d'en déduire quelques unes des conséquences les plus directes concernant les ensembles de rotation. Avant cela, j'énoncerai deux des corollaires de cette théorie qui me semblent parmi les plus emblématiques et ferai un détour didactique par un énoncé de forçage en dimension 1, qui motivera la définition de fer à cheval topologique.

\section{Quelques applications aux homéomorphismes de la sphère}

Les applications de la théorie du forçage sur les surfaces impressionnent par leur nombre et leur puissance : outre les théorèmes numérotés de A à M dans \cite{LCT1} et de A à L dans \cite{LCT2}, on peut citer les travaux \cite{conejeros2016}, \cite{conejeros2019} et \cite{Gabriel} se basant de manière cruciale sur cette théorie. Elles concernent à la fois la théorie de la rotation sur l'anneau, le tore et les surfaces de genre supérieur, mais aussi les homéomorphismes non errants de la sphère, les difféomorphismes hamiltoniens du tore, etc.

Plutôt que de faire la liste de toutes ces applications, je voudrais en présenter deux qui me semblent parmi les plus frappantes ; nous rentrerons plus dans le détail des applications aux homéomorphismes du tore dans la dernière partie de cet exposé. \'Etant donnée une surface $S$, on notera $\homeo_0(S)$ l'ensemble des homéomorphismes de $S$ isotopes à l'identité.

La première est un théorème de structure des ensembles non errants d'homéomorphismes de la sphère d'entropie nulle. Rappelons que l'\emph{ensemble non errant} d'un homéomorphisme $f$ d'un espace compact $M$ est défini par
\[\Omega(f) = \big\{x\in M \mid \forall U \text{ vois. de } x,\exists n>0 : f^n(U) \cap U \neq\emptyset\big\}.\]
Dans cet ensemble on peut sélectionner les points ayant un comportement asymptotique \og non trivial\fg{} : on pose\footnote{Pour tout point $x\in S$, on définit son \emph{$\omega$-limite} (resp. $\alpha$-limite) comme l'ensemble des points d'accumulation de l'orbite positive (resp. négative) de $x$ sous $f$.}
\[\Omega'(f) = \big\{z\in \Omega(f) \mid \alpha(z) \cup \omega(z) \not\subset  \fix(f)\big\}.\]
Notons que l'ensemble des points récurrents non fixes de $f$ est un $G_\delta$ dense de $\Omega'(f)$.

\cite{MR3033517} ont démontré un théorème de structure pour les difféomorphismes de la sphère préservant le volume, d'entropie nulle et de classe $C^\infty$. Leur preuve est essentiellement topologique mais utilise (entres autres) la théorie de Yomdin sur le taux de croissance de la longueur des arcs, qui requiert la régularité $C^\infty$. Ce théorème a été étendu successivement dans \cite{LCT1} aux homéomorphismes non errants (théorème M), et dans \cite{LCT2} (théorème G) à l'ensemble non errant d'un homéomorphisme.

\begin{theo}\label{ThSphere}
Soit $f\in \homeo_0(\Sp^2)$ un homéomorphisme de la sphère préservant l'orientation, et sans fer à cheval topologique\footnote{Cette notion sera définie au paragraphe suivant. Le lecteur connaissant le fer à cheval de Smale peut avoir cet exemple en tête ; mentionnons le fait qu'un homéomorphisme d'entropie nulle est sans fer à cheval topologique.}. Alors l'ensemble $\Omega'(f)$ est recouvert par une famille $(A_\alpha)_{\alpha\in\mathcal A}$ d'anneaux ouverts tels que pour tout $\alpha\in\mathcal A$,
\begin{enumerate}
\item $f|_{A_\alpha}$ est sans point fixe et isotope à $\id_{A_\alpha}$ ;
\item 
il existe un relevé de $f|_{A_\alpha}$ au revêtement universel de $A_\alpha$ dont l'ensemble de rotation est inclus dans $[0,1]$ ;
\item $A_\alpha$ est minimal pour les deux premières propriétés.
\end{enumerate}
De plus, si l'homéomorphisme $f$ est non errant (i.e. $\Omega(f) = \Sp^2$), alors les anneaux $A_\alpha$ sont deux-à-deux disjoints et leur union est dense dans $\Sp^2\setminus\fix(f)$.
\end{theo}

\begin{figure}
\begin{center}
\includegraphics[width=.6\linewidth]{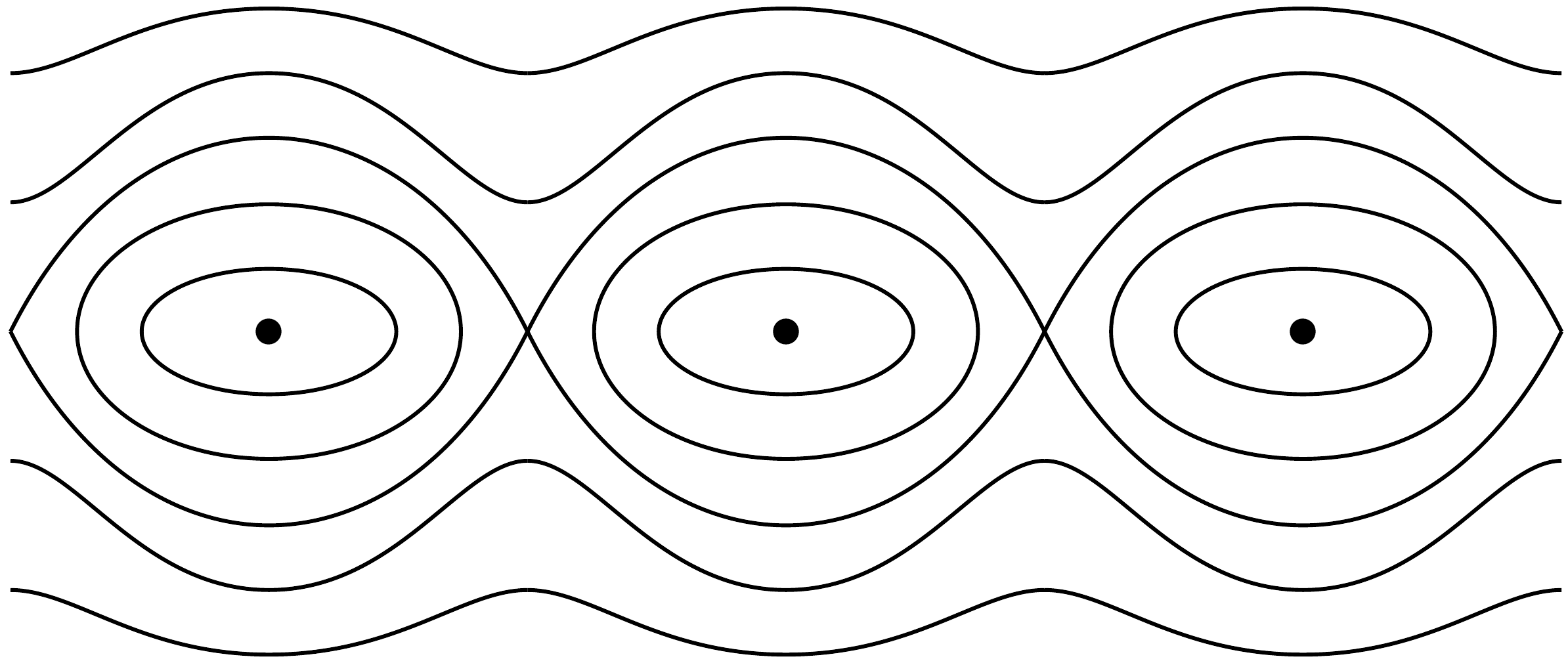}
\caption{\label{FigPorPhase}Portrait de phase du type pendule simple.}
\end{center}
\end{figure}

Remarquons tout de suite que cet énoncé a des liens très intimes avec le théorème \ref{ThPB} de Poincaré-Birkhoff, à travers la condition 2 qui exprime que les anneaux, sans points fixes, ne peuvent pas vérifier les conditions du théorème de Poincaré-Birkhoff.

Le théorème \ref{ThSphere} exprime que la dynamique d'un homéomorphisme de la sphère d'entropie nulle \og ressemble\fg\footnote{Attention tout de même, la situation peut être bien plus complexe que sur la figure, en particulier il n'y a pas de raison pour que l'homéomorphisme possède des courbes invariantes.} à celle du portrait de phase du pendule simple (figure \ref{FigPorPhase}). Notons qu'il est possible que l'ensemble $\mathcal A$ soit vide (et donc qu'il n'y ait aucun anneau), par exemple pour une dynamique du type nord-sud.

Le résultat est plus faible dans le cas général d'un ensemble errant non vide (on ne sait pas si les anneaux sont disjoints) : de l'aveu même des auteurs il reste encore du travail pour obtenir une description complète des homéomorphismes de la sphère d'entropie nulle. Néanmoins, la preuve de ce dernier théorème est longue et délicate dans le cas général (elle occupe une part importante de \cite{LCT2}) : nous ne l'aborderons pas ici.
\medskip

La seconde application concerne les ensembles transitifs. On dit qu'un ensemble $\Lambda\subset \Sp^2$ est \emph {transitif} pour $f$ s'il existe une orbite de $f|_{\Lambda}$ dense dans $\Lambda$. C'est une propriété plus forte que le fait d'être non errant, ce qui permet d'avoir une description plus précise dans le cas des homéomorphismes d'entropie nulle. Commençons par deux définitions.

\begin{defi}
Nous dirons que $f : \Sp^2\to\Sp^2$ est \emph{topologiquement infiniment renormalisable} pour un ensemble fermé invariant non vide $\Lambda$ s'il existe une suite croissante d'entiers $(q_n)_{n\ge 1}$ et une suite décroissante $(D_n)_{n\ge 1}$ de disques topologiques ouverts tels que :
\begin{enumerate}
\item $q_n$ divise $q_{n+1}$ ;
\item $f^{q_n}(D_n) = D_n$ ;
\item les disques $f^i(D_n)$, pour $0\le i <q_n$, sont deux-à-deux disjoints ;
\item $\Lambda \subset \bigcup_{0\le i <q_n} f^i(D_n)$.
\end{enumerate}
\end{defi}

Notons que si $f$ est topologiquement infiniment renormalisable pour $\Lambda$, alors $f|_{\Lambda}$ est semi-conjuguée à un odomètre (voir les figures \ref{FigOdo} et \ref{FigOdo2}).

\begin{figure}
\begin{minipage}[c]{0.49\textwidth}
\begin{center}
\begin{tikzpicture}[scale=2.5]
\draw[->] (-.2,0) -- (1.2,0);
\draw[->] (0,-.2) -- (0,1.2);
\draw (1,0) node{$|$};
\draw (1,-.08) node[below]{$1$};
\draw (0,1) node{$-$};
\draw (-.08,1) node[left]{$1$};
\foreach \x in {-5,...,0}{
\draw[very thick](1-2^\x,0.5*2^\x)  -- (1-0.5*2^\x,2^\x);
}
\end{tikzpicture}
\caption{Une réalisation de l'odomètre par une application de l'intervalle.}\label{FigOdo}
\end{center}
\end{minipage}
\hfill
\begin{minipage}[c]{0.49\textwidth}
\begin{center}
\includegraphics[width=.7\linewidth]{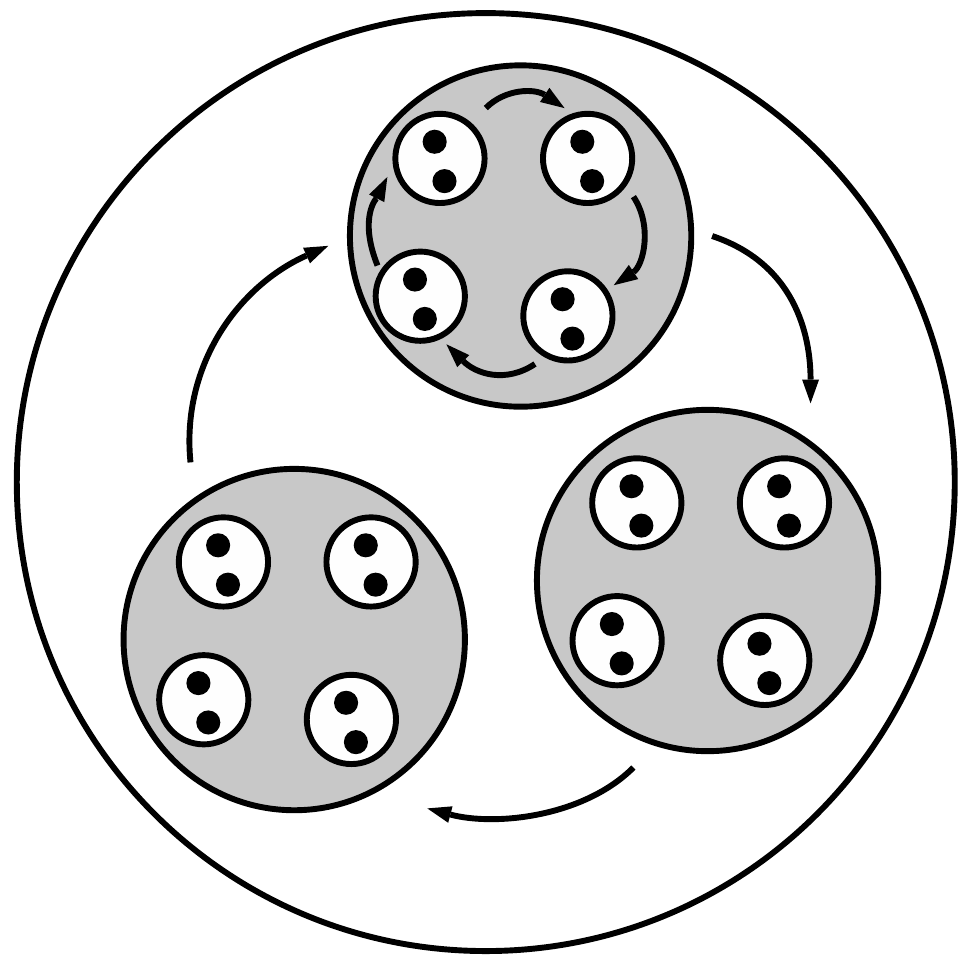}
\caption{\label{FigOdo2}Un exemple d'application renormalisable.}
\end{center}
\end{minipage}
\end{figure}

\begin{defi}
Nous dirons qu'un ensemble fermé invariant non vide $\Lambda$ est \emph{de type irrationnel} s'il est l'union disjointe d'un nombre fini d'ensembles fermés $\Lambda_1,\dots,\Lambda_q$ qui sont permutés cycliquement par $f$, et s'il existe un nombre irrationnel $\rho\in \R$ tel que pour tout $n\ge 1$, si $z_0 \neq z_1$ sont deux points fixes de $f^{qn}$, alors il existe un relevé $\check g$ de $f^{qn}|_{\Sp^2\setminus\{z_0,z_1\}}$ au revêtement universel de $\Sp^2\setminus\{z_0,z_1\}$ tel que $\check g|_{\check\Lambda}$ possède un unique nombre  de rotation, qui est égal soit à $n\rho$, soit à 0.
\end{defi}

Bien sûr, l'exemple le plus simple d'ensemble de type irrationnel est un cercle où la dynamique est de nombre de rotation irrationnel, mais il est possible de construire d'autres exemples plus complexes, par exemple par la méthode d'Anosov-Katok qui donne des pseudo-rotations irrationnelles transitives sur la sphère.

Le résultat suivant est la proposition K de \cite{LCT2}.

\begin{theo}
Soit $f\in \homeo_0(\Sp^2)$ un homéomorphisme de la sphère préservant l'orientation, sans fer à cheval topologique, et $\Lambda$ un ensemble invariant transitif et non vide. Alors
\begin{itemize}
\item soit $\Lambda$ est une orbite périodique ;
\item soit $f$ est topologiquement infiniment renormalisable pour $\Lambda$ ;
\item soit $\Lambda$ est de type irrationnel.
\end{itemize}
\end{theo}

\section{Un cas d'école : la dimension 1}

Il me semble éclairant de commencer par revenir aux fondamentaux et expliquer un exemple de théorème de forçage d'orbites en dimension 1 : le théorème de \cite{Sarko}. Sa preuve, bien que très simple, sera éclairante pour la suite de notre exposition.

\begin{theo}\label{ThSharko}
Soit $f : \R \to\R$ une application continue, qui possède un point périodique de période 3\footnote{C'est-à-dire un point $x\in\R$ tel que $f^3(x) = x$ mais $f^k(x)\neq x$ pour tout $1\le k< 3$.}. Alors $f$ possède des points périodiques de n'importe quelle période entière.
\end{theo}

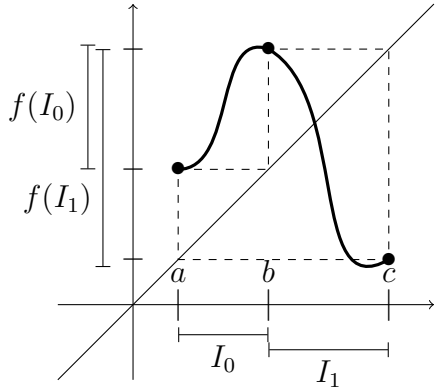
\begin{figure}
\begin{center}
\begin{tikzpicture}[scale=2]
\draw[->] (-.5,0) -- (2,0);
\draw[->] (0,-.5) -- (0,2);
\draw (.3,0) node{$|$};
\draw (.3,.08) node[above]{$a$};
\draw (.9,0) node{$|$};
\draw (.9,.08) node[above]{$b$};
\draw (1.7,0) node{$|$};
\draw (1.7,.08) node[above]{$c$};
\draw (0,.3) node{$-$};
\draw (0,.9) node{$-$};
\draw (0,1.7) node{$-$};
\draw[very thick] (.3,.9) node{$\bullet$} to[out=0, in=160] (.9,1.7) node{$\bullet$} to[out=-30, in=-150] (1.7,.3) node{$\bullet$};
\draw (-.5,-.5) -- (2,2);
\draw[dashed] (.3,.9) -- (.9,.9) -- (.9,1.7) -- (1.7,1.7) -- (1.7,.3) -- (.3,.3) -- cycle;
\draw[<->,>=|] (.3,-.2) --node[below]{$I_0$} (.9,-.2);
\draw[<->,>=|] (.9,-.3) --node[below]{$I_1$} (1.7,-.3);
\draw[<->,>=|] (-.3,.9) --node[left]{$f(I_0)$} (-.3,1.73);
\draw[<->,>=|] (-.2,.25) -- (-.2,1.7);
\draw(-.2,.7) node[left]{$f(I_1)$};
\end{tikzpicture}
\caption{Preuve du théorème \ref{ThSharko} de Sharkovsky.}\label{FigSharko}
\end{center}
\end{figure}

\begin{proof}
Soit $a$ un point périodique de période 3 pour $f$. On note $b = f(a)$ et $c = f^2(a)$ (et donc $f(c) = a$). Quitte à permuter circulairement $a$, $b$ et $c$, et à changer $f$ en $x\mapsto -f(-x)$, on peut supposer que $a<b<c$. Posant $I_0 = ]a,b[$ et $I_1 = ]b,c[$, le théorème des valeurs intermédiaires assure que $f(I_0)\supset ]b,c[ = I_1$ et $f(I_1) \supset ]a,c[ \supset I_0\cup I_1$ (voir la figure \ref{FigSharko}). Autrement dit, définissant la relation $\to_n$ sur les intervalles de $[a,c]$ par $I\to_n J$ si $f^n(I)\supset J$, on a $I_0 \to_1 I_1$, $I_1\to_1 I_0$ et $I_1\to_1 I_1$. On peut représenter ces relations par le graphe suivant, où les flèches correspondent aux relations $\to_1$ :
\begin{center}
\begin{tikzpicture}[scale=.9]
\node[draw,circle,minimum height=1.2cm, fill=gray!20!white] (A) at (3,0) {$I_1$};
\node[draw,circle,minimum height=1.2cm, fill=gray!20!white] (B) at (0,0) {$I_0$};
\draw[->,>=latex,shorten >=3pt, shorten <=3pt] (B) to[bend left] (A);
\draw[->,>=latex,shorten >=3pt, shorten <=3pt] (A) to[bend left] (B);
\draw[->,>=latex,shorten >=3pt, shorten <=3pt] (A) to[loop right,looseness=8,min distance=10mm] (A);
\end{tikzpicture}
\end{center}

Cette relation $\to_n$ vérifie les propriétés
\begin{itemize}
\item de transitivité : si $I\to_n J$ et $J\to_m K$, alors $I\to_{m+n} K$ ;
\item de point fixe : par le théorème des valeurs intermédiaires, si $I\to_n I$, alors il existe $x\in \overline I$ fixe par $f^n$.
\end{itemize}
Par conséquent, si on choisit une trajectoire positive $w_0 w_1 w_2\dots\in\{0,1\}^\N$ dans le graphe ci-dessus --- ce qui signifie que pour tout $j\ge 0$, on a $I_{w_j} \to_1 I_{w_{j+1}}$ ---, alors il existe un point $x\in\R$ tel que $f^j(x) \in I_{w_j}$ pour tout $j\ge 0$. Si de plus la trajectoire est périodique de période $T$ qui n'est pas un multiple de 3, alors on peut montrer facilement que $x$ peut être pris périodique de période $T$.
\end{proof}

Dans la preuve, on vient de définir un ensemble de mots \emph{admissibles}, définis par un nombre fini de règles de transition locales. Celles-ci peuvent être lues sur la matrice d'adjacence 
\[A = \begin{pmatrix}
 1 & 1 \\ 1 & 0
\end{pmatrix}\]
du graphe, auquel est aussi naturellement associée une chaîne de Markov. L'application de la dynamique $f$ correspond à l'application décalage \og\emph{shift}\fg{} sur l'espace des mots infinis :
\[\begin{array}{rcl}
\sigma : \{0,1\}^\N & \longrightarrow & \{0,1\}^\N \\
 w_0w_1w_2\dots & \longmapsto & w_1w_2w_3\dots
\end{array}\]
Formellement, si on pose $\Sigma$ l'ensemble des mots admissibles, et $\Lambda = \overline{\bigcap_{n\in\N} f^{-n}\big(I_0\cup I_1\big)}$, alors on a
\begin{center}
\begin{tikzpicture}[scale=.9]
\node (A) at (0,0) {$\Sigma$};
\node (B) at (2,0) {$\Sigma$};
\node (C) at (0,2) {$\Lambda$};
\node (D) at (2,2) {$\Lambda$};

\draw[->,>=latex,shorten >=3pt, shorten <=3pt] (A) to node[midway, above]{$\sigma$} (B);
\draw[->,>=latex,shorten >=3pt, shorten <=3pt] (C) to node[midway, left]{$h$} (A);
\draw[->,>=latex,shorten >=3pt, shorten <=3pt] (C) to node[midway, above]{$f$} (D);
\draw[->,>=latex,shorten >=3pt, shorten <=3pt] (D) to node[midway, right]{$h$} (B);
\end{tikzpicture}
\end{center}
où $h$ est défini par $f^i(x) \in I_{h(x)_i}$ (à une indétermination près, en le point 3 périodique). Par ce qu'on vient de voir, l'application $h$ est surjective et continue, autrement dit $f_{|\Lambda}$ est une \emph{extension} du --- ou est semi-conjugué au --- décalage de Fibonacci $\sigma_{|\Sigma}$. De plus, à chaque point périodique de $\sigma$  correspond un point périodique de même période dans $\overline\Lambda$.

Derrière cette semi-conjugaison se cache l'idée d'un codage de la dynamique, dont les origines remontent (au moins) à Hadamard et sa preuve de l'ergodicité du flot géodésique en 1898.
\medskip

La configuration du théorème \ref{ThSharko} de Sharkovsky est l'exemple type de \emph{fer à cheval topologique} comme défini\footnote{À quelques détails près : la définition qui vient est adaptée aux applications inversibles et aux réelles partitions de Markov et pas seulement aux partitions topologiques.} dans \cite{LCT2} (voir \cite{MR660643} pour une étude plus complète des fers à cheval en dimension 1). La définition qui suit, bien qu'un peu longue, est un critère très général sous lequel la dynamique d'un décalage se retrouve au sein de la dynamique de l'application considérée.

\begin{defi}\label{DefCheval}
Soit $X$ un espace localement compact, et $f$ un homéomorphisme de $X$. Un \emph{fer à cheval topologique} est un sous-ensemble compact $Y$ de $X$ qui est invariant par un itéré $f^r$ de $f$, tel qu'il existe un ensemble compact $Z$, une application continue $g : Z\to Z$, un entier $q\ge 2$ et des applications continues surjectives $h_1 : Z \to Y$ et $h_2 : Z \to \{1,\dots, q\}^\Z$ telles que le diagramme suivant commute
\begin{center}
\begin{tikzpicture}[scale=1.3]
\node (A) at (0,0) {$Y$};
\node (B) at (2.5,0) {$Y$};
\node (C) at (0,1.5) {$Z$};
\node (D) at (2.5,1.5) {$Z$};
\node (E) at (0,3) {$\{1,\dots, q\}^\Z$};
\node (F) at (2.5,3) {$\{1,\dots, q\}^\Z$};

\draw[->,>=latex,shorten >=3pt, shorten <=3pt] (A) to node[midway, above]{$f$} (B);
\draw[->,>=latex,shorten >=3pt, shorten <=3pt] (C) to node[midway, left]{$h_1$} (A);
\draw[->,>=latex,shorten >=3pt, shorten <=3pt] (C) to node[midway, above]{$g$} (D);
\draw[->,>=latex,shorten >=3pt, shorten <=3pt] (D) to node[midway, right]{$h_1$} (B);
\draw[->,>=latex,shorten >=3pt, shorten <=3pt] (E) to node[midway, above]{$\sigma$} (F);
\draw[->,>=latex,shorten >=3pt, shorten <=3pt] (C) to node[midway, left]{$h_2$} (E);
\draw[->,>=latex,shorten >=3pt, shorten <=3pt] (D) to node[midway, right]{$h_2$} (F);
\end{tikzpicture}
\end{center}
et que 
\begin{itemize}
\item les fibres de $h_1$ soient toutes finies ;
\item la préimage de toute suite $p$-périodique de $\sigma$ contient une suite $p$-périodique de $g$.
\end{itemize}
\end{defi}

En particulier, l'existence d'un fer à cheval topologique implique la positivité de l'entropie topologique\footnote{Il existe une réciproque à ce résultat en régularité $C^{1+\alpha}$, démontrée par \cite{MR573822}, qui est fausse en régularité $C^0$, par une constructions due à \cite{MR616561}.} (un nombre qui mesure le caractère chaotique du système) et la croissance exponentielle avec $p$ du nombre d'orbites périodiques de période $p$.
\medskip

En résumé, l'idée d'une théorie du forçage est de détecter des mécanismes impliquant l'existence de nouvelles orbites dans notre dynamique, voire si c'est possible l'existence de fers à cheval topologiques. Dans le cas des surfaces, l'exemple emblématique de fer à cheval topologique est le fer à cheval de Smale (voir la figure \ref{FigSmale}), pour lequel l'ensemble $Y$ de la définition \ref{DefCheval} est $\bigcap_{n\in\Z} f^n(R)$. Celui-ci est semi-conjugué au décalage complet sur un alphabet à deux éléments, et cela vient du fait que l'ensemble image $f(R)$ possède deux \emph{intersections markoviennes} avec le rectangle $R$. Cette notion remplace celle (bien plus rudimentaire) de recouvrement en dimension 1, et là où on appliquait un théorème des valeurs intermédiaires pour trouver un point périodique, il faut ici utiliser un théorème d'indice (de Lefschetz ou de Conley). Dans cet exposé, nous ne rentrerons pas dans le détail de la preuve d'existence de fer à cheval topologique, nous nous contenterons d'observer l'apparition de configurations géométriquement similaires à celle du fer à cheval de Smale.

\begin{figure}
\begin{center}
\begin{tikzpicture}[scale=.4, rotate=90]
\draw (0,0) rectangle (6,6);
\fill[opacity=.15] (0,0) rectangle (6,6);
\draw (1,-1) -- (1,7) arc (180:0:2) -- (5,-1) -- (4,-1) -- (4,7) arc (0:180:1) -- (2,-1) -- cycle;
\fill[opacity=.15] (1,-1) -- (1,7) arc (180:0:2) -- (5,-1) -- (4,-1) -- (4,7) arc (0:180:1) -- (2,-1) -- cycle;
\draw (3,3) node{$R$};
\draw (3,9) node[left]{$f(R)$};
\end{tikzpicture}
\caption{Le fer à cheval de Smale}\label{FigSmale}
\end{center}
\end{figure}
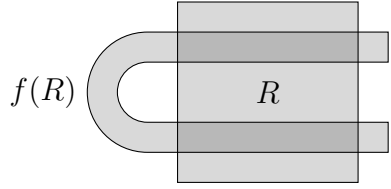

\section{Prérequis : théorie de Brouwer et feuilletages}

Dans toute la suite le plan $\R^2$ sera orienté. On appellera \emph{chemin} toute application continue d'un intervalle dans une surface, et \emph{droite} tout plongement topologique propre de $\R$ dans $\R^2$. Le complémentaire de toute droite $\Delta$ possède deux composantes connexes (via le théorème de Jordan appliqué à la compactification $\Sp^2$ de $\R^2$). Le choix d'une (co-)orientation sur $\Delta$ permet de lever l'ambiguïté sur ces composantes connexes, qui seront désignées par la \emph{gauche} $L(\Delta)$ et la \emph{droite} $R(\Delta)$ de $\Delta$. 

\begin{defi}
Une \emph{droite de Brouwer} pour un homéomorphisme $f$ du plan est une droite topologique orientée $\Delta$ proprement plongée dans $\R^2$, telle que $\overline{f(R(\Delta))} \subset R(\Delta)$.
\end{defi}

Lorsque $f$ préserve l'orientation, cela implique que pour tous $m<n\in\Z$, la droite $f^n(\Delta)$ est située à droite de $f^m(\Delta)$. Le théorème de translation plane de \cite{MR1511684} énonce que par chaque point passe une droite de Brouwer.

\begin{theo}[Brouwer]\label{ThBrou}
Soit $f$ un homéomorphisme du plan préservant l'orientation et sans point fixe. Alors en tout point du plan passe une droite de Brouwer pour $f$.
\end{theo}

Donnons tout de suite l'application directe la plus frappante de ce théorème, qui peut être vue comme un premier résultat de forçage dans le plan\footnote{Je triche un peu en appelant l'énoncé qui suit corollaire, puisqu'il est un résultat intermédiaire des preuves du théorème de Brouwer.}.

\begin{coro}
Tout homéomorphisme du plan préservant l'orientation et ayant un point périodique possède aussi un point fixe.
\end{coro}

\begin{proof}
Il suffit de montrer qu'un homéomorphisme préservant l'orientation et sans point fixe est aussi sans point périodique. Soit $p$ un point du plan, le théorème de Brouwer \ref{ThBrou} assure l'existence d'une droite de Brouwer passant par $p$, en particulier cette droite est libre (disjointe de chacun de ses itérés) et donc $p$ n'est pas périodique.
\end{proof}

Toutes les preuves du théorème de Brouwer sont assez longues et délicates (mais se terminent par un calcul d'indice de courbe). Vu la puissance de ce résultat, il est naturel de chercher à l'améliorer, dans au moins deux directions différentes.
\begin{itemize}
\item Il n'y a aucune raison pour que deux droites de Brouwer issues de deux points distincts soient disjointes. Est-il possible de faire en sorte que ce soit le cas, de manière à obtenir un feuilletage du plan par des droites de Brouwer ?
\item Est-il possible d'obtenir un tel résultat pour des homéomorphismes de surfaces compactes ? Autrement dit, étant donné un homéomorphisme d'une surface compacte, à quelles conditions existe-t-il un feuilletage de la surface qui se relève, au revêtement universel, en un feuilletage du plan en droites de Brouwer pour le relevé de l'homéomorphisme ?
\end{itemize}

Ces deux questions ont été résolues successivement par \cite{PatriceFeuille1} et \cite{PatriceFeuille2} (principalement, voir aussi l'exposé \cite{MR2275671}), pour aboutir à un théorème de Brouwer feuilleté équivariant.

\begin{theo}\label{TheoPatriceFeuillete}
Soit $G$ un groupe discret d'homéomorphismes du plan préservant
l'orientation, agissant librement et proprement. Soit $f$ un homéomorphisme du plan sans point fixe, préservant l'orientation et commutant avec les éléments de $G$. Alors il existe un feuilletage topologique orienté du plan en droites de Brouwer pour $f$, et invariant sous l'action de $G$.
\end{theo}

Dans la pratique, on appliquera ce théorème uniquement via son corollaire suivant, qui utilise la notion de chemin transverse.

\begin{defi}
Soit $\F$ un feuilletage orienté d'une surface $S$ et $\gamma : [0,1]\to S$ un chemin. Pour $x\in S$, on note $\phi_x$ la feuille de $\F$ passant par $x$. On dit que $\gamma$ est \emph{positivement transverse} à $\F$ (qu'on abrégera en $\F$-transverse) si pour tout $t\in[0,1]$, on a en relevant au revêtement universel\footnote{Ce revêtement universel est toujours homéomorphe au plan, puisqu'il n'existe pas de feuilletage non singulier sur la sphère.} de $S$
\[\{\wt\gamma(u)\mid u\in [0,t[\}\subset L(\wt \phi_{\wt \gamma(t)}) \qquad \text{et} \qquad \{\wt \gamma(u)\mid u\in ]t,1]\}\subset R(\wt \phi_{\wt \gamma(t)}).\]
\end{defi}

Autrement dit, le chemin $\gamma$ traverse localement toutes les feuilles qu'il rencontre de la gauche vers la droite.

\begin{coro}\label{CoroPatriceFeuillete}
Soit $S$ une surface et $I=(f_t)_{t\in[0,1]}$ une isotopie dans $S$ joignant l'identité à un homéomorphisme $f$ (si bien que pour tout $z\in S$, l'arc $I(z): t \mapsto f_t(z)$ joint $z$ à $f(z)$). On suppose que $I$ n'a pas de point fixe contractile, c'est-à-dire de point fixe $z$ tel que $I(z)$ soit un lacet homotope à $0$. Alors il existe un feuilletage topologique orienté $\F$ sur $S$ et pour tout $z\in S$, un chemin positivement transverse à $\F$ joignant $z$ à $f(z)$ qui est homotope, à extrémités fixées, à l'arc $I(z)$.
\end{coro}

On notera $I_\F(z)$ un chemin $\F$-transverse homotope à extrémités fixes à $I(z)$ (dont on déduit la définition des chemins $I^n_\F(z)$).

Ce corollaire se déduit du théorème feuilleté facilement ; nous en donnons une preuve, basée sur le lemme suivant.

\begin{lemm}\label{LemExisTraj}
Sous les hypothèses du théorème \ref{TheoPatriceFeuillete}, pour tout point $z_0\in\R^2$, il existe un arc positivement transverse à $\F$ qui joint $z_0$ à $f(z_0)$.
\end{lemm}

\begin{figure}
\begin{center}
\begin{tikzpicture}[scale=.8]
\fill[opacity=.1] (0,2) -- (0,-2) -- (2,-2) to[in=180, out=90] (3,-.5) -- (3,1) to[out=180, in=-90] (1,2) -- cycle;
\draw (0,2) -- (0,-2) node[below]{$\phi_0$};
\draw (0,0) node {$\times$} node[left]{$z_0$};
\draw (1,2) to[in=180, out=-90] (3,1)node[right]{$\phi_1$};
\draw (2,-2) to[in=180, out=90] (3,-.5)node[right]{$\phi_2$};
\draw (1,0) node{$W$};
\draw[->] (-.3,-.5) -- (.2,-.5);
\draw[->] (1.5,1.5) -- (1.3,1);
\draw[->] (2.3,-1.5) -- (1.9,-1.3);
\end{tikzpicture}
\caption{Preuve du lemme \ref{LemExisTraj}. Les feuilles sont co-orientées. Le lecteur notera que les conventions d'orientation des feuilles seront parfois différentes de celles utilisées dans \cite{LCT1}.\label{FigLemExisTraj}}
\end{center}
\end{figure}
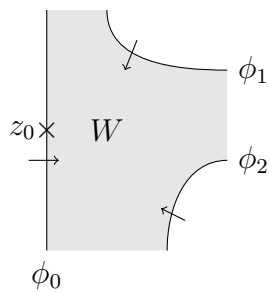

\begin{proof}
Notons $W$ l'ensemble des points de $\R^2$ pouvant être joints à $z_0$ par un arc positivement transverse à $\F$ (et non trivial). Alors $W$ est un ensemble ouvert, dont la frontière est constituée de la feuille $\phi_0$ passant par $z_0$ ainsi que d'éventuelles autres feuilles $\phi$ vérifiant $W\subset R(\phi)$ (voir la figure \ref{FigLemExisTraj}). Si $f(z_0)$, qui est à droite de $\phi_0$, n'appartenait pas à $W$, alors il appartiendrait à l'adhérence de la gauche d'une de ces feuilles $\phi$, ce qui est impossible puisque ces ensembles sont invariants par $f^{-1}$.
\end{proof}

\begin{proof}
Le corollaire se déduit immédiatement du lemme par passage au revêtement universel de $S\setminus \fix(f)$. En effet, ce revêtement universel ne peut être la sphère vue l'hypothèse sur l'absence de point fixe contractile, et est donc homéomorphe au plan $\R^2$. L'hypothèse sur l'absence de point fixe contractile assure que tout relevé de $f$ est sans point fixe. On applique alors le théorème de feuilletage équivariant à ce relevé de $f$ et au groupe du revêtement.
\end{proof}

Le corollaire \ref{CoroPatriceFeuillete} demande une hypothèse assez forte : $I$ ne doit pas avoir de point fixe contractile. Le théorème de Brouwer feuilleté ne prend toute sa force que lorsqu'on lui joint un résultat d'existence d'isotopies, dites \emph{maximales}, qui satisfont cette condition, démontré récemment par \cite{bguin2016fixed}, et qui implique le résultat suivant\footnote{L'article \cite{LCT1} se base sur le travail \cite{MR3329669}, qui donne un résultat plus faible et moins naturel, mais suffisant en pratique, concernant les isotopies maximales.}. Pour une isotopie $I$ entre homéomorphismes, on note $\fix(I) = \bigcap_{t\in[0,1]} \fix(f_t)$ l'ensemble des points dont l'orbite sous $I$ est entièrement fixe, et $\dom(I) = S\setminus \fix(I)$ le \emph{domaine} de l'isotopie $I$.

\begin{theo}\label{ThExistIstop}
Soit $f\in \homeo_0(S)$. Alors il existe une isotopie $I$ entre $f$ et l'identité telle que $I|_{\dom(I)}$ n'ait pas de point fixe contractile.
\end{theo}

Autrement dit, pour tout $z\in \fix(f)\setminus \fix(I)$, la trajectoire $I(z)$ n'est pas contractile dans $\dom(I)$. Ce théorème est crucial, puisqu'il permet, pour tout homéomorphisme homotope à l'identité, d'appliquer le corollaire \ref{CoroPatriceFeuillete} à la surface $\dom(I)$ et à l'isotopie $I|_{\dom(I)}$.
\medskip

Le corollaire \ref{CoroPatriceFeuillete} permet d'associer, à chaque trajectoire $I(z): t \mapsto f_t(z)$ d'une isotopie sans point fixe contractile, un arc positivement transverse qui lui est homotope à extrémités fixes. Cet arc est en fait uniquement déterminé par l'ensemble des feuilles de $\wt \F$ traversées par la trajectoire $\wt I(z)$ relevée au revêtement universel traversées par la trajectoire $\wt I(z)$ relevée au revêtement universel $\wt \dom(I)$ de $\dom(I)$ ; cela amène à considérer les trajectoires des isotopies dans l'espace des feuilles, autrement dit à quotienter par la relation sur les arcs \og rencontrer les mêmes feuilles\fg{} : la théorie du forçage va en fait nous permettre de détecter de nouvelles orbites dans cet espace quotient des feuilles.

\begin{defi}
Soit $S$ une surface, $\F$ un feuilletage de $S$, $f$ un homéomorphisme de $S$ homotope à l'identité et $I$ une isotopie reliant l'identité à $f$. Pour $n\in\N$ et $z\in S$, notons $I^n(z)$ la concaténation des arcs $I(f^j(z))$ pour $0\le j<n$. On dit qu'un chemin $\F$-transverse $\gamma : [a,b]\to S$ est \emph{admissible d'ordre $n\in\N$} si au revêtement universel, $\tilde f^n\big(\tilde\phi_{\tilde\gamma(a)}\big) \cap \tilde \phi_{\tilde\gamma(b)} \neq\emptyset$.
\end{defi}

\section{Apparition des fers à cheval}

Nous arrivons au c{\oe}ur de la théorie du forçage : on va commencer par énoncer la proposition fondamentale \ref{PropFond} de \cite{LCT1}, qui permet de trouver, lorsqu'apparaît une \emph{intersection $\F$-transverse}, de nouvelles orbites dans l'espace des feuilles : en gros, si deux orbites sous l'isotopie se croisent dans l'espace des feuilles, alors n'importe quel choix de direction fait à cet embranchement sera lui aussi réalisé par des orbites de l'isotopie (voir la figure \ref{FigEspFeuilles}). La preuve de cette proposition est élémentaire et repose seulement sur l'idée d'ensembles séparant le plan.

\begin{figure}
\begin{minipage}[c]{0.4\textwidth}
\begin{center}
\includegraphics[width=.7\linewidth]{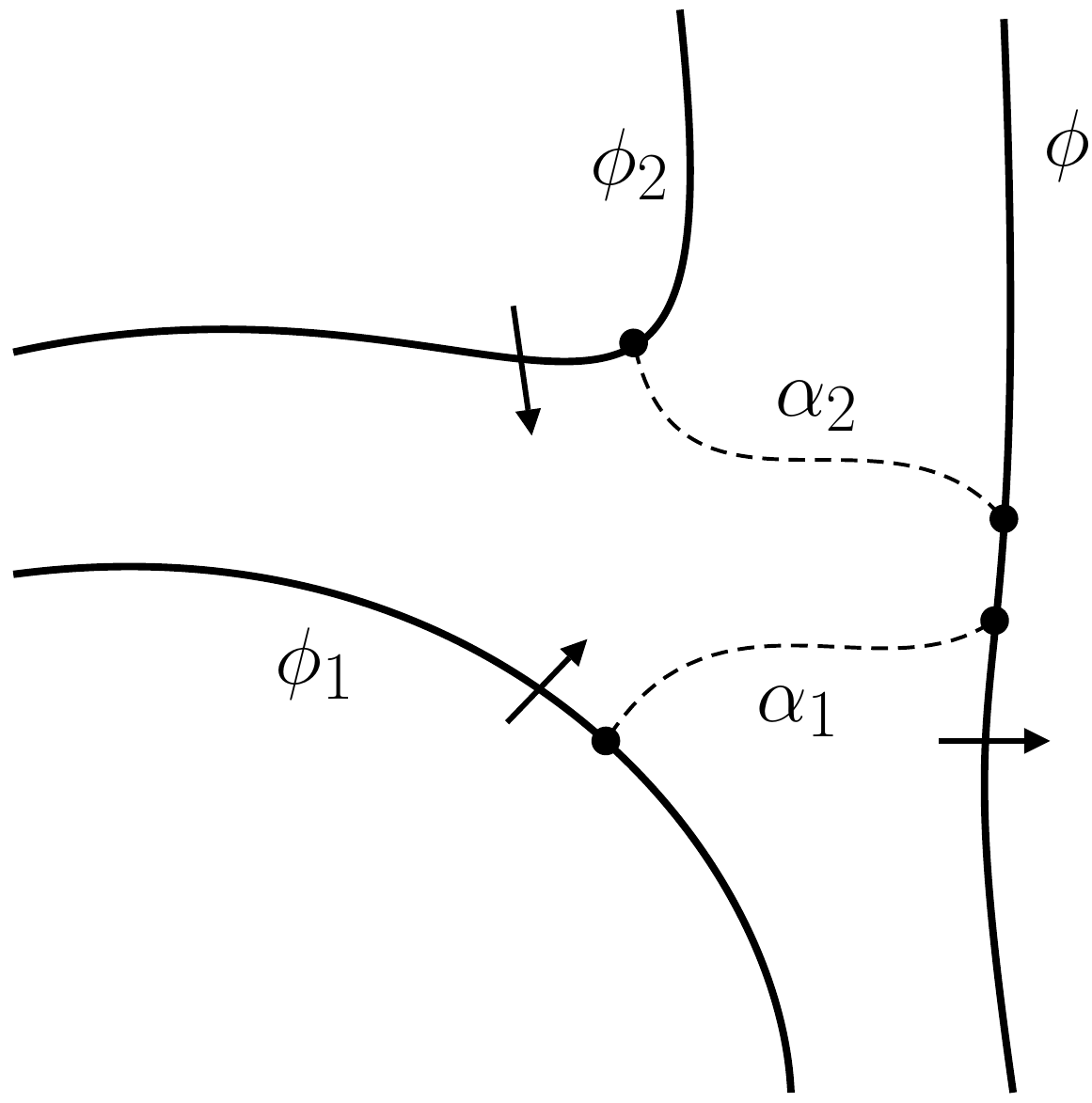}
\caption{\label{FigDessus}Suivant la définition \ref{DefDessus}, la feuille orientée $\phi_2$ est au-dessus de la feuille orientée $\phi_1$ relativement à $\phi$.}
\end{center}
\end{minipage}
\hfill
\begin{minipage}[c]{0.59\textwidth}
\begin{center}
\begin{tikzpicture}[scale=.75]
\draw[thick] (-3,1.5) -- (-1,.1) -- (1,.1) -- (3,-1.5);
\draw[thick] (-3,-1.5) -- (-1,-.1) -- (1,-.1) -- (3,1.5);
\draw[thick, dashed] (-3,1.7) -- (-1,.3) -- (1,.3) -- (3,1.7);
\draw[thick, dashed] (-3,-1.7) -- (-1,-.3) -- (1,-.3) -- (3,-1.7);
\end{tikzpicture}
\caption{Proposition fondamentale et trajectoires dans l'espace des feuilles : si les deux trajectoires en trait plein sont admissibles, alors les deux trajectoires en pointillés le sont aussi.}\label{FigEspFeuilles}
\end{center}
\end{minipage}
\end{figure}

Cette proposition est à la base de la preuve de l'existence de fers à cheval dans le cas d'une \emph{auto-intersection $\F$-transverse} (lorsque l'intersection transverse vient d'une seule et même orbite de l'isotopie) comme démontrée dans \cite{LCT2}.

Commençons par définir rigoureusement la notion d'intersection $\F$-transverse.

\begin{defi}\label{DefDessus}
Soient $\phi$, $\phi_1$ et $\phi_2$ trois droites orientées du plan. Nous dirons que \emph{$\phi_2$ est au-dessus de $\phi_1$ relativement à $\phi$} si (voir la figure \ref{FigDessus}) :
\begin{itemize}
\item ces trois droites sont deux à deux disjointes ;
\item aucune des deux ne sépare les deux autres ;
\item si $\alpha_i$, $i=1,2$ sont deux chemins disjoints reliant chacun un point de $\phi_i$ à un point $\phi(t_i)$ de $\phi$\footnote{On utilise un paramétrage des feuilles, compatible avec leur (co-)orientation.}, alors $t_2>t_1$.
\end{itemize}
\end{defi}

Soit $\F$ un feuilletage du plan, et pour $i=1,2$, $J_i$ deux intervalles réels et $\gamma_i = J_i\to\R^2$ deux chemins $\F$ transverses.

\begin{defi}\label{DefInterTrans}
On dit que $\gamma_1 : J_1\to\R^2$ et $\gamma_2 : J_2\to\R^2$ \emph{s'intersectent $\F$-transversalement et positivement\footnote{Par la suite, il nous arrivera d'omettre le \og positivement \fg.}} s'il existe $a_i<t_i< b_i\in J_i$ tels que (voir la figure \ref{FigInterTransvers}) :
\begin{itemize}
\item $\phi_{\gamma_1(t_1)} = \phi_{\gamma_2(t_2)} = \phi$ ;
\item $\phi_{\gamma_2(a_2)}$ est au-dessus de $\phi_{\gamma_1(a_1)}$ relativement à $\phi$ ;
\item $\phi_{\gamma_1(b_1)}$ est au-dessus de $\phi_{\gamma_2(b_2)}$ relativement à $\phi$.
\end{itemize} 
\end{defi}

\begin{figure}
\begin{minipage}[c]{0.54\textwidth}
\begin{center}
\includegraphics[width=\linewidth]{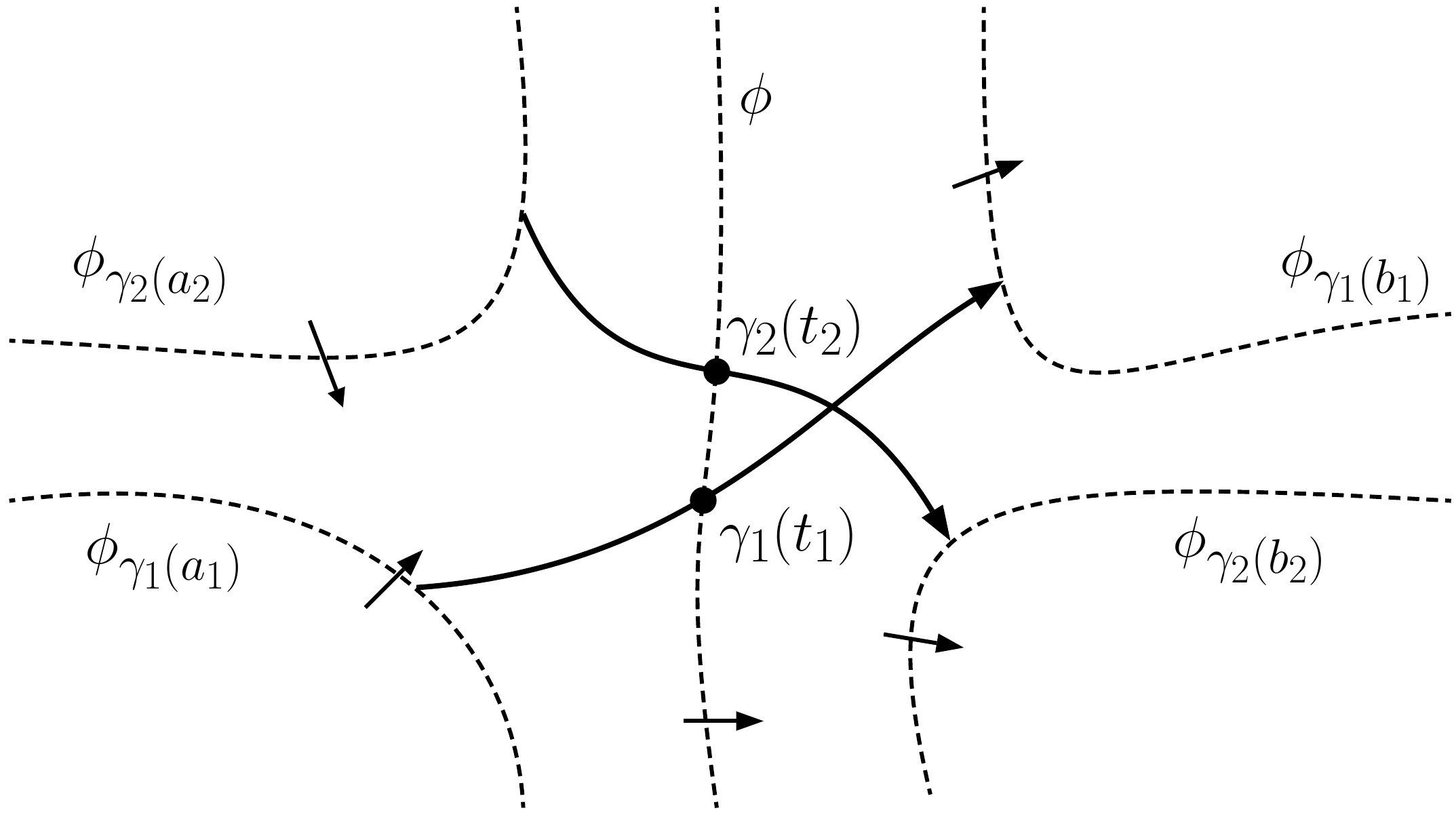}
\caption{\label{FigInterTransvers}Suivant la définition \ref{DefInterTrans}, les chemins $\gamma_1$ et $\gamma_2$ s'intersectent $\F$-transversalement et positivement.}
\end{center}
\end{minipage}
\hfill
\begin{minipage}[c]{0.45\textwidth}
\begin{center}
\includegraphics[width=\linewidth]{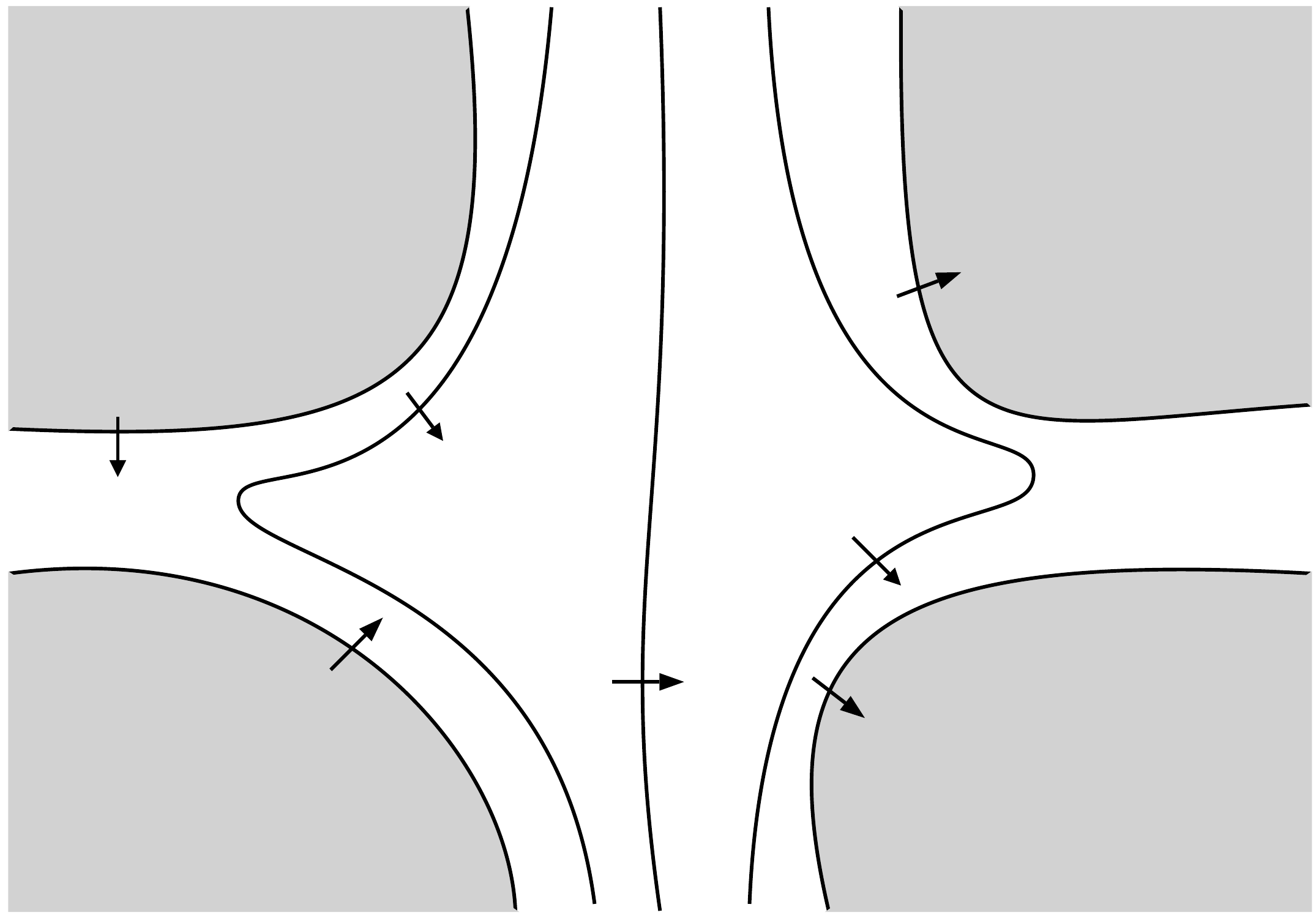}
\caption{\label{FigPropFond1}Configuration possible des feuilles dans la proposition fondamentale.}
\end{center}
\end{minipage}
\end{figure}

Remarquons qu'en cas d'intersection $\F$-transverse, quitte à modifier $\gamma_1$ ou $\gamma_2$ dans leur classe d'équivalence, on peut supposer que $\gamma_1(t_1) = \gamma_2(t_2)$ (ce qu'on fera par la suite).	

Voici la proposition fondamentale de la théorie du forçage (proposition 20 de \cite{LCT1}).

\begin{prop}\label{PropFond}
Soit $f$ un homéomorphisme du plan et $\F$ un feuilletage de Brouwer pour $f$.
On suppose que $\gamma_i : [a_i,b_i] \to \R^2$ ($i=1,2$) s'intersectent $\F$-transversalement en $\gamma_1(t_1) = \gamma_2(t_2)$. Si chaque $\gamma_i$ est admissible d'ordre $n_i$, alors les chemins $\gamma_1|_{[a_1,t_1]}\gamma_2|_{[t_2,b_2]}$ et $\gamma_2|_{[a_2,t_2]}\gamma_1|_{[t_1,b_1]}$ sont tous deux admissibles d'ordre $n_1+n_2$. De plus, soit ces deux chemins sont admissibles d'ordre $\max(n_1,n_2)$, soit au moins l'un des deux est admissible d'ordre $\min(n_1,n_2)$.
\end{prop}

\begin{proof}
L'idée de la preuve est résumée par la figure \ref{FigPropFond2}. Remarquons que, parce que les feuilles de $\F$ sont des droites de Brouwer, et par l'hypothèse que $\phi_{\gamma_2(a_2)}$ est au-dessus de $\phi_{\gamma_1(a_1)}$ relativement à $\phi$, on a, pour tout $k\ge 0$, 
\[f^{-k}(\overline{L \phi_{\gamma_1(a_1)}}) \cap \overline{L \phi_{\gamma_2(a_2)}}
\subset \overline{L \phi_{\gamma_1(a_1)}} \cap \overline{L \phi_{\gamma_2(a_2)}}
= \emptyset.\]
Plus généralement, pour tous $k_1,k_2\in\Z$,
\begin{equation}\label{EqPasInter}
f^{k_1}(\overline{L \phi_{\gamma_1(a_1)}}) \cap f^{k_2}(\overline{L \phi_{\gamma_2(a_2)}}) = f^{k_1}(\overline{R \phi_{\gamma_1(b_1)}}) \cap f^{k_2}(\overline{R \phi_{\gamma_2(b_2)}}) = \emptyset.
\end{equation}

Pour $i=1,2$, définissons les ensembles
\[X_i = f^{n_i}(\overline{L \phi_{\gamma_i(a_i)}}) \cup \overline{R \phi_{\gamma_i(b_i)}} \qquad \text{et} \qquad Y_i = f^{-n_i}(\overline{R \phi_{\gamma_i(b_i)}}) \cup \overline{L \phi_{\gamma_i(a_i)}}.\]
Par hypothèse sur l'admissibilité des chemins, ces quatre ensembles sont connexes.

Si $\gamma_2|_{[a_2,t_2]}\gamma_1|_{[t_1,b_1]}$ n'est pas admissible d'ordre $n_2$, alors $f^{n_2}(\phi_{\gamma_2(a_2)}) \cap \phi_{\gamma_1(b_1)} = \emptyset$, si bien que par \eqref{EqPasInter}, $X_2 \cap R\phi_{\gamma_1(b_1)} = \emptyset$.

\begin{figure}
\begin{center}
\includegraphics[width=.6\linewidth]{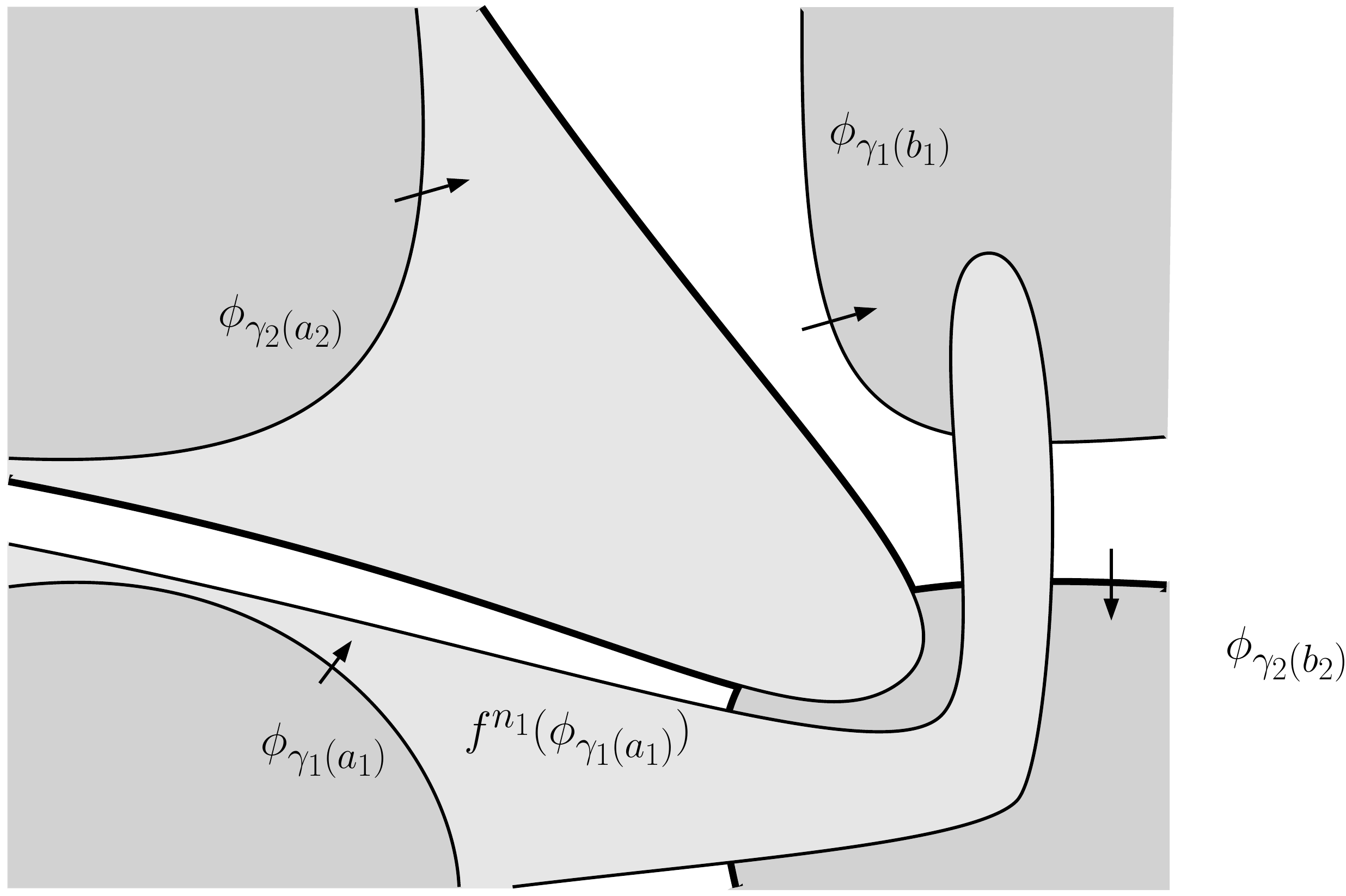}
\caption{\label{FigPropFond2}Proposition fondamentale : si $f^{n_2}(\phi_{\gamma_2(a_2)})$ ne rencontre pas $\phi_{\gamma_1(b_1)}$, cela force l'intersection entre $f^{n_1}(\phi_{\gamma_1(a_1)})$ et $\phi_{\gamma_2(b_2)}$. La frontière de l'ensemble $X_2$ est représenté en traits gras.}
\end{center}
\end{figure}

\begin{lemm}\label{LemPropFond}
L'ensemble $X_2$ sépare $\overline{L\phi_{\gamma_1(a_1)}}$ et $\overline{R\phi_{\gamma_1(b_1)}}$.
\end{lemm}

\begin{proof}
Considérons un chemin $\alpha_1 : [0,1]\to \R^2$ reliant $\overline{L\phi_{\gamma_1(a_1)}}$ à $\overline{R\phi_{\gamma_1(b_1)}}$, ainsi que $\alpha_2 : [0,1]\to \R^2$ un chemin reliant $\overline{L\phi_{\gamma_2(a_2)}}$ à $\overline{R\phi_{\gamma_2(b_2)}}$ et inclus dans $X_2$. Ces deux chemins rencontrent la feuille centrale $\phi$ en respectivement $u_1$ et $u_2$. Si $\alpha_1$ et $\alpha_2$ étaient disjoints, le fait que l'intersection est transverse impliquerait que d'une part $u_1<u_2$ et d'autre part $u_2<u_1$. Cette contradiction implique que les chemins $\alpha_1$ et $\alpha_2$ s'intersectent.
\end{proof}

Le lemme \ref{LemPropFond} implique que chacune des intersections $X_1 \cap X_2$ et $Y_1\cap X_2$ est non vide. La première propriété implique que $f^{n_1}(\phi_{\gamma_1(a_1)}) \cap \phi_{\gamma_2(b_2)} \neq \emptyset$, autrement dit que
$\gamma_1|_{[a_1,t_1]}\gamma_2|_{[t_2,b_2]}$ est admissible d'ordre $n_1$. La seconde implique que $f^{n_2}(\phi_{\gamma_2(a_2)}) \cap f^{-n_1}(\phi_{\gamma_1(b_1)}) \neq \emptyset$, autrement dit que
$\gamma_2|_{[a_2,t_2]}\gamma_1|_{[t_1,b_1]}$ est admissible d'ordre $n_1+n_2$.

Un raisonnement similaire (où $Y_1$ est utilisé pour séparer) montre que si $\gamma_2|_{[a_2,t_2]}\gamma_1|_{[t_1,b_1]}$ n'est pas admissible d'ordre $n_1$, alors il est admissible d'ordre $n_1+n_2$, et $\gamma_1|_{[a_1,t_1]}\gamma_2|_{[t_2,b_2]}$ est admissible d'ordre $n_2$.

En résumé, $\gamma_2|_{[a_2,t_2]}\gamma_1|_{[t_1,b_1]}$ est admissible d'ordre $n_1+n_2$, et s'il n'est pas admissible d'ordre $\max(n_1,n_2)$, alors l'autre chemin $\gamma_1|_{[a_1,t_1]}\gamma_2|_{[t_2,b_2]}$ est admissible d'ordre $\min(n_1,n_2)$. On conclut par symétrie des rôles joués par $\gamma_1$ et $\gamma_2$.
\end{proof}

Nous allons maintenant en déduire l'existence d'un fer à cheval topologique dans le cas d'une auto-intersection $\F$-transverse.

Soit $S$ une surface orientée et $f\in \homeo_0(S)$. Soit $I$ une isotopie maximale entre $f$ et l'identité, donnée par le théorème \ref{ThExistIstop}, ainsi que $\F$ le feuilletage transverse de $\dom(I)$ pour $f_{|\dom(I)}$ donné par le corollaire \ref{CoroPatriceFeuillete}. On définit $\widetilde \dom(I)$ le revêtement universel de $\dom(I)$, et $\widetilde f$, $\widetilde I$ et $\widetilde \F$ les relevés de $f|_{\dom(I)}$, $I_{|\dom(I)}$ et $\F$ à $\widetilde \dom(I)$.

Voici le théorème M de \cite{LCT2}, qui est une amélioration significative des résultats de \cite{LCT1}, où les auteurs n'obtenaient que la positivité de l'entropie et la croissance exponentielle du nombre de points périodiques, sous l'hypothèse plus forte que le point $z$ est récurrent.

\begin{theo}\label{ThExistCheval}
Supposons qu'il existe $z\in \widetilde \dom(I)$, $q\ge 2$ et $T$ un automorphisme de revêtement de $\widetilde \dom(I)$ tels que les trajectoires $\widetilde I_\F^q(z)$ et $T\widetilde I_\F^q(z)$ aient une intersection $\wt \F$-transverse. Alors $f$ possède un fer à cheval topologique, et $h(f)\ge (\log 4)/3q$.
\end{theo}

Le fer à cheval obtenu est dit \emph{rotationnel}, propriété qui permet d'en déduire des propriétés concernant l'ensemble de rotation de l'homéomorphisme.

La preuve de ce théorème est relativement longue et technique, nous pouvons tout de même en donner les idées principales.

\begin{figure}
\begin{center}
\includegraphics[width=\linewidth]{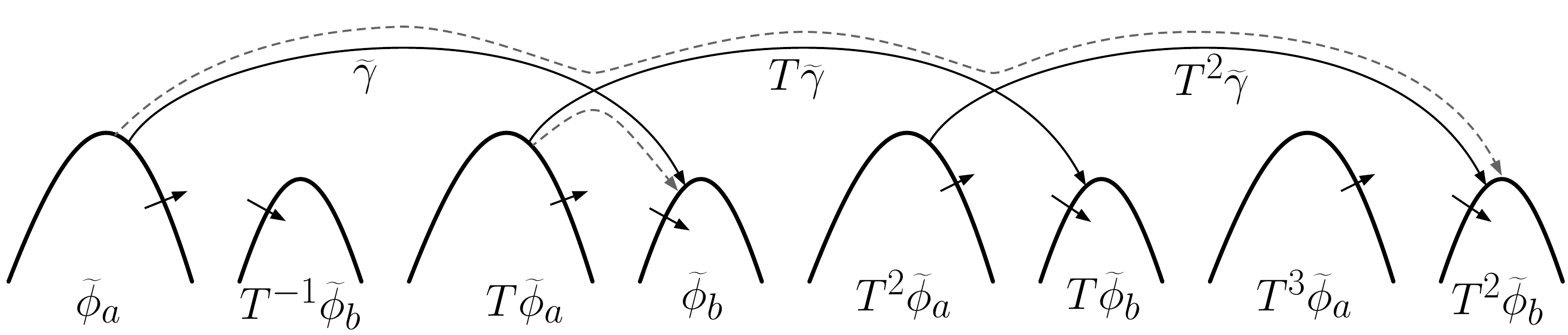}
\caption{\label{FigCheval1}Utilisation de la proposition fondamentale \ref{PropFond} pour trouver de nouvelles orbites (en gris pointillé) dans la preuve du théorème \ref{ThExistCheval}.}
\end{center}
\end{figure}

\begin{proof}[Idées de preuve]
Notons $\wt\gamma : [a,b]\to \widetilde \dom(I)$ la trajectoire $\wt I_\F^q(z)$. Par hypothèse, il existe $a<s<t<b$ tels que $\wt\gamma$ et $T\wt\gamma$ ont une intersection $\wt \F$-transverse en $T\wt\gamma(s) = \wt\gamma(t)$. Pour simplifier, nous noterons $\wt\phi_a = \wt\phi_{\wt\gamma(a)}$ et $\wt\phi_b = \wt\phi_{\wt\gamma(b)}$, et supposerons $q\ge 2$.

Pour commencer, on démontre l'existence de nouvelles orbites dans l'espace des feuilles, assurée par la proposition fondamentale.

\begin{lemm}\label{LemExistCheval1}
Pour tout $q\ge 2$ et tout $p\in\{0,\cdots,q\}$, on a $\wt f^q(\wt\phi_a) \cap T^{p-1}\wt\phi_b \neq\emptyset$.
\end{lemm}

L'idée de ce lemme est résumée par la figure \ref{FigCheval1} ; la preuve consiste à appliquer successivement la proposition fondamentale \ref{PropFond} aux translatés de $\wt\gamma$ (en fait on utilise le corollaire 22 de \cite{LCT1}).

Ce lemme implique sans trop de difficulté une version faible du théorème \ref{ThExistCheval}, qui énonce que pour tout $q\ge 2$ et $p\in\{1,\dots,q\}$, il existe un point fixe à $T^{-p}\wt f^q$ (proposition 12 de \cite{LCT2}). 

Voyons comment démontrer cette implication. Pour $p \ge 0$ et $q\ge 1$, on considère les ensembles $\mathcal X_{p,q}$ collectant tous les sous-chemins de $T^p \wt f^{-q}(\wt\phi_b)$ reliant $\wt\phi_a$ à $T\wt\phi_a$ et disjoints des gauches de tous les $T^k\wt\phi_a$ (voir la figure \ref{FigCheval3}) ; plus formellement $\mathcal X_{p,q}$ est l'ensemble des chemins reliant $\wt\phi_a$ à $T\wt\phi_a$, qui sont composante connexe de
\[T^p \wt f^{-q}(\wt\phi_b) \cap \bigcap_{k\in\Z}R\big( T^k\wt\phi_a\big).\]
\'Etudions le cas $p\le q-1$. Le lemme \ref{LemExistCheval1} assure que la courbe $T^p\wt f^{-q}(\wt\phi_b)$ rencontre à la fois $\wt\phi_a$ et $T\wt\phi_a$. Alors un bout de la courbe $T^p\wt f^{-q}(\wt\phi_b)$ joignant $\wt\phi_a$ à $T\wt\phi_a$ et disjoint de $L(\wt\phi_a) \cup L(T\wt\phi_a)$ ne rencontre aucun des $L(T^k\wt\phi_a)$ (autrement dit $T^p\wt f^{-q}(\wt\phi_b)$ doit rencontrer les $L(T^k\wt\phi_a)$ les uns à la suite des autres) ; cela vient du fait que la projection de courbe $T^p\wt f^{-q}(\wt\phi_b)$ sur $\dom(I)$ est simple\footnote{Moralement, si cela n'était pas vrai, en quotient par $T$, on obtiendrait une lacet simple dans l'anneau $\wt\dom(I)/T$ de valeur au moins 2 en homologie. Pour les détails, voir le lemme 10 de \cite{LCT2}.}. On vient de montrer que l'ensemble $\mathcal X_{p,q}$ est non vide. De la même manière, $\mathcal X_{p+1,q}$ est lui aussi non vide.

\begin{figure}
\begin{minipage}[c]{0.49\textwidth}
\includegraphics[width=\linewidth]{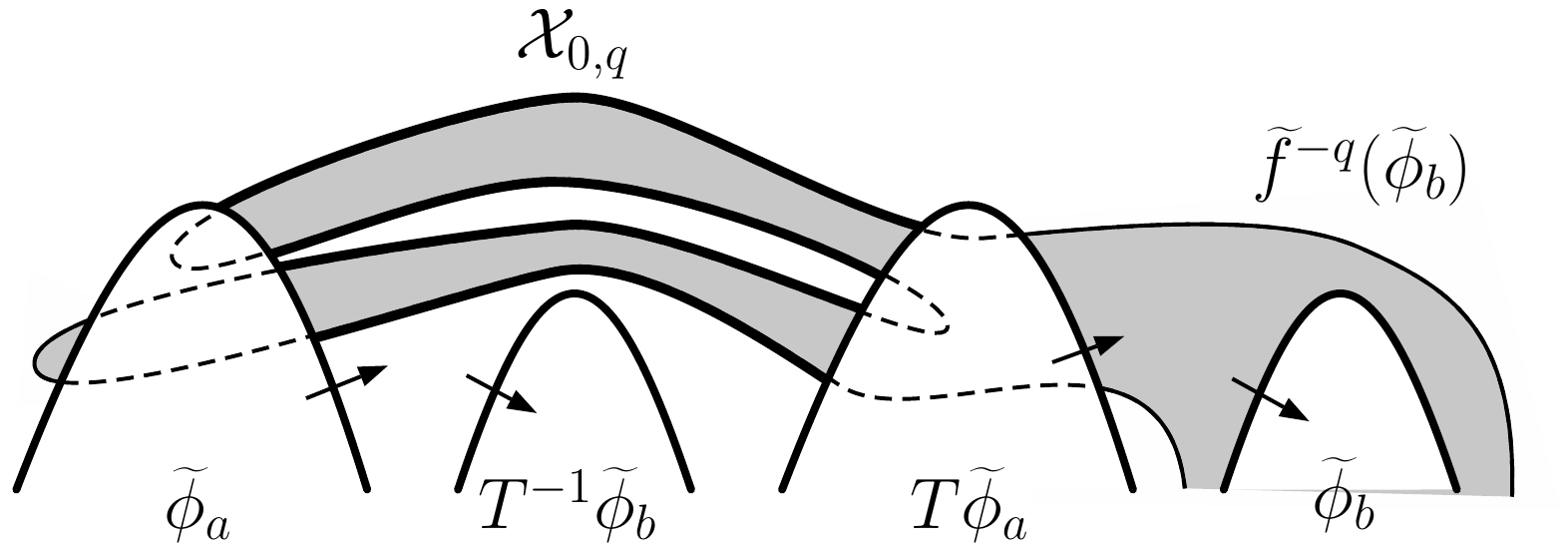}
\caption{\label{FigCheval3}L'ensemble $\mathcal X_{0,q}$ dans la preuve du théorème \ref{ThExistCheval}.}
\end{minipage}
\hfill
\begin{minipage}[c]{0.49\textwidth}
\includegraphics[width=\linewidth]{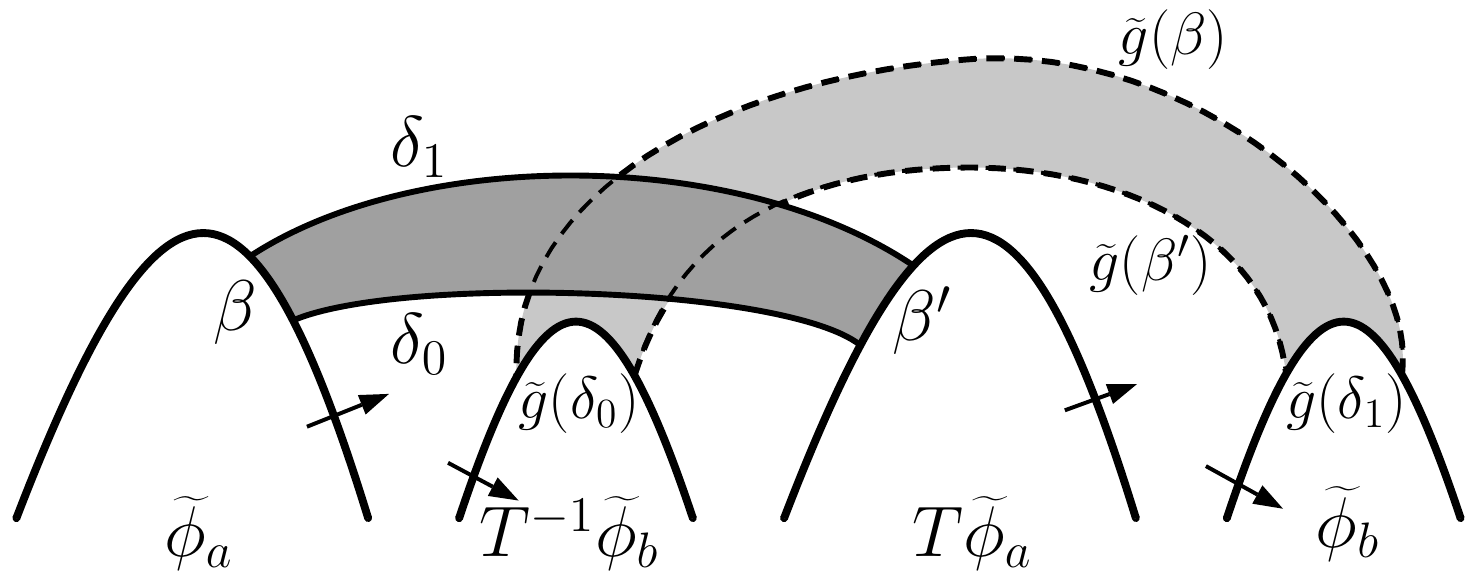}
\caption{\label{FigCheval2}Apparition de l'intersection markovienne dans la preuve du théorème \ref{ThExistCheval}.}
\end{minipage}
\end{figure}

On considère alors $\delta_0\in \mathcal X_{p,q}$ reliant $\wt\phi_a(t_0)$ à $T\wt\phi_a(t'_0)$ et $\delta_1\in \mathcal X_{p+1,q}$ reliant $\wt\phi_a(t_1)$ à $T\wt\phi_a(t'_1)$. On pose aussi $\beta = \wt\phi_{a|[t_0,t_1]}$ et $\beta' = T\wt\phi_{a|[t'_0,t'_1]}$ (voir la figure \ref{FigCheval2}). Visuellement, le rectangle formé par les quatre courbes $\delta_0$, $\delta_1$, $\beta$ et $\beta'$ semble avoir une intersection markovienne avec son image par $\wt g = T^{-p}\wt f^q$. La vérification de ce fait est un peu fastidieuse, nous renvoyons au calcul de l'indice de Lefschetz de la courbe faite dans \cite{LCT2}. Ce calcul implique l'existence d'un point fixe à $\wt g$.
\medskip

L'existence du fer à cheval demande un peu plus de travail, mais les idées sont similaires à celles qu'on vient de voir. Si $q\ge 3$ et $1<p<q$, alors on peut montrer, par des arguments topologiques simples, que les ensembles $\mathcal X_{p-1,q}$ et $\mathcal X_{p+1,q}$ sont non vides, et qu'entre deux courbes $\delta_1\in \mathcal X_{p-1,q}$ et $\delta_4\in \mathcal X_{p+1,q}$, il existe au moins deux courbes $\delta_2,\delta_3\in \mathcal X_{p,q}$. Ces quatre courbes permettent de définir deux rectangles, comme sur la figure \ref{FigCheval4}, l'image de chacun intersectant de manière markovienne chacun des deux autres. Lorsqu'on considère le rectangle obtenu en remplissant ces deux petits rectangles (i.e. le rectangle situé entre $\delta_1$ et $\delta_4$), on obtient un fer à cheval topologique. Là encore, les vérifications sont un peu fastidieuses mais ne font pas vraiment intervenir de nouvelles idées. Pour justifier rigoureusement l'existence d'un tel fer à cheval, les auteurs utilisent la théorie de l'indice de Conley ; il est néanmoins très probable qu'un raisonnement basé sur l'invariant plus facile à manipuler qu'est l'indice de Lefschetz d'une courbe permette d'aboutir aux mêmes résultats.
\end{proof}

\begin{figure}
\begin{center}
\includegraphics[width=.7\linewidth]{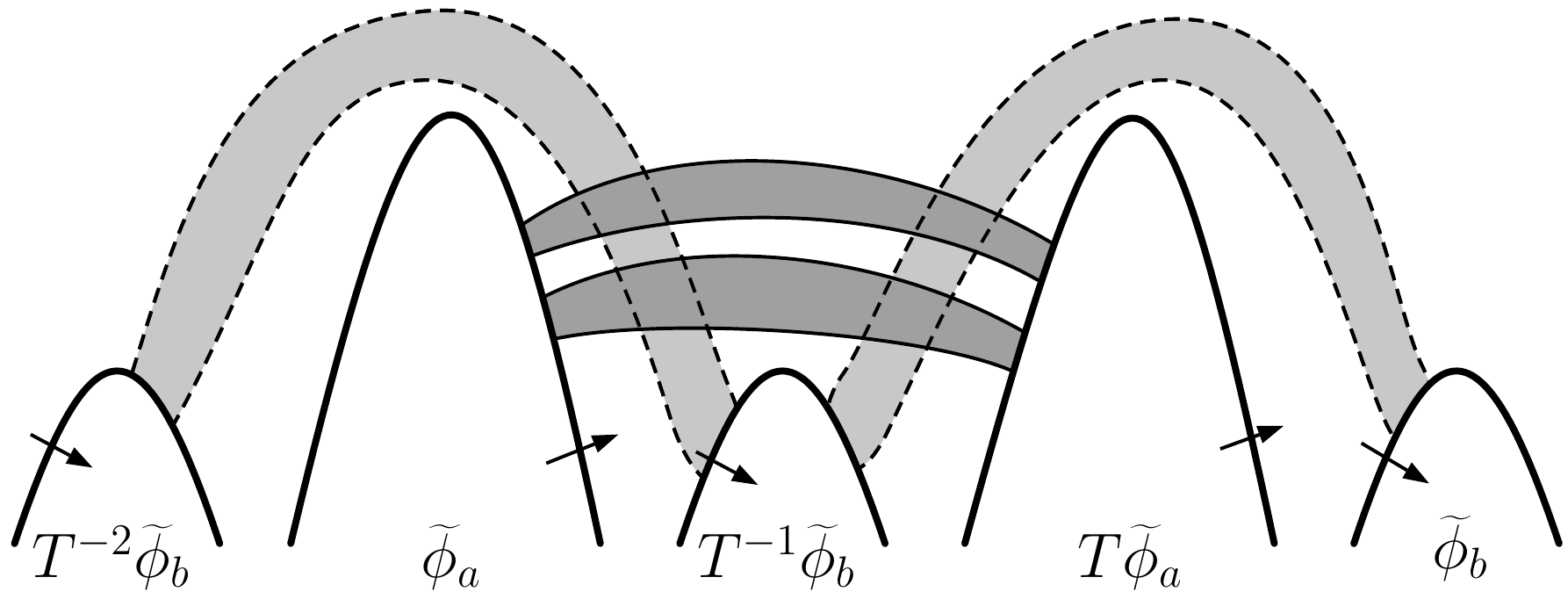}
\caption{\label{FigCheval4}Apparition du fer à cheval.}
\end{center}
\end{figure}

\section{Conséquences pour les ensembles de rotation}

Pour illustrer la puissance de la théorie du forçage, nous allons expliquer quelques-unes de ses applications à l'étude de l'ensemble de rotation d'un homéomorphisme du tore homotope à l'identité. C'est un invariant dynamique généralisant le nombre de rotation de Poincaré pour les homéomorphismes du cercle, qui mesure à quelle vitesse moyenne les orbites tournent autour du tore. Contrairement au cas du cercle, où il y a des restrictions dues à la présence d'un ordre naturel sur les triplets de points, il n'y a aucune raison pour que ces vitesses soient les mêmes pour tous les points : on obtient en général un ensemble et non un unique vecteur.

Plus formellement, pour $f\in\homeo_0(\T^2)$ et $\check f$ un relevé de $f$ à $\R^2$, on définit
\[\rho(\check f) = \bigcap_{N\in\N}\overline{\bigcup_{n>N}\left\{\frac{\check f^n(z)-z}{n} \mid z\in \R^2\right\}}.\]
On voit très facilement que c'est un sous-ensemble compact non vide de $\R^2$, indépendant du choix du relevé de $f$ à translation de $\Z^2$ près, invariant par conjugaison dans $\homeo_0(\T^2)$ et vérifiant, pour tout $n\in\N^*$, $\rho(\check f^n) = n\rho(\check f)$. Il est un peu plus difficile de voir que cet ensemble est convexe\footnote{C'est par ailleurs une propriété spécifique à la dimension 2.}, ce qui est démontré dans \cite{MR1053617}. La question qui arrive naturellement est de savoir quelles propriétés dynamiques de l'homéomorphisme peuvent être lues sur cet invariant. Par exemple, le théorème d'existence d'orbites périodiques pour un homéomorphisme du cercle de nombre de rotation rationnel se généralise au tore de dimension 2.

\begin{theo}[\cite{MR958891}]\label{ThFranks}
Soit $f\in \homeo_0(\T^2)$ et $\check f$ un relevé de $f$ à $\R^2$. Alors pour tout $p/q\in\inte (\rho(\check f))$ écrit de manière irréductible, il existe un point $z\in\R^2$ tel que $\check f^q(z) = z+p$.
\end{theo}

La première application de la théorie du forçage aux ensembles de rotation que nous décrirons résout un cas dans une conjecture de \cite{MR1021217} ; c'est le premier résultat décrivant un sous-ensemble convexe compact du plan qui n'est l'ensemble de rotation d'aucun homéomorphisme du tore (théorème 62 de \cite{LCT1}).

\begin{theo}\label{ThMZ}
Soit $f\in \homeo_0(\T^2)$ et $\check f$ un relevé de $f$ à $\R^2$. Alors $\partial \rho(\check f)$ ne contient aucun segment à pente irrationnelle ayant un point rationnel dans son intérieur.
\end{theo}

La seconde application concerne les homéomorphismes du tore ayant un ensemble de rotation d'intérieur non vide. Il retrouve le résultat de \cite{MR1101087} concernant la positivité de l'entropie et améliore ceux de \cite{MR3355114} (traitant la régularité $C^{1+\alpha}$) et \cite{MR3835526} (traitant les ensemble de rotation polygonaux). C'est une combinaison des théorèmes 63 et 64 de \cite{LCT1}.

\begin{theo}\label{TheoRhoInt}
Soit $f\in \homeo_0(\T^2)$ et $\check f$ un relevé de $f$ à $\R^2$. On suppose que $\rho(\check f)$ est d'intérieur non vide. Alors
\begin{itemize}
\item $f$ possède un fer à cheval topologique (et donc est d'entropie topologique positive), et
\item il existe $L>0$ tel que pour tout $z\in\R^2$ et tout $n\ge 1$, on ait $d\big(\check f^n(z)-z,\, n\rho(\check f)\big) \le L$.
\end{itemize}
\end{theo}

Notons que la théorie du forçage a justement été construite sur les idées du travail de \cite{MR3835526}.

Le second point du théorème donne une vitesse de convergence vers l'ensemble de rotation. Plus précisément, l'ensemble de rotation peut être défini de la manière alternative suivante : étant donné un domaine fondamental $D\subset \R^2$ du tore, l'ensemble de rotation est la limite au sens de Hausdorff des ensembles $\check f^n(D)/n$. Le second point du théorème exprime que cette convergence se fait à vitesse au pire $O(1/n)$. Malheureusement, la constante $L$ n'est pas vraiment explicite (elle dépend de l'isotopie maximale) et ce théorème ne peut donc pas être utilisé pour des applications au calcul explicite de l'ensemble de rotation.

Terminons les commentaires sur le théorème \ref{TheoRhoInt} en notant qu'il admet un corollaire presque immédiat concernant le nombre de rotation d'une mesure de probabilité invariante $\mu$, qui est défini comme (voir \cite{MR88720})\footnote{On utilise le fait que l'application $\R^2\ni x \mapsto \check f(x)-x$ est $\Z^2$-équivariante et passe donc au quotient $\T^2$.}
\[\rho(\check f, \mu) = \int_{\T^2} \big(\check f(x)-x\big)  dx.\]
Ce corollaire énonce que si l'homéomorphisme, en plus d'avoir un ensemble de rotation d'intérieur non vide, préserve une mesure de probabilité de support total $\mu$, alors le nombre de rotation de $\mu$ est à l'intérieur de l'ensemble de rotation. Cela constitue une réponse positive à une question de Boyland.
\medskip

Les preuves de ces deux derniers théorèmes sont relativement simples comparativement aux autres de la théorie du forçage, parce que dans ce cas les feuilles du feuilletage sont bornées dans le plan, ce qui permet de trouver facilement des intersections transverses, en fait dès que deux orbites ont des directions rotationnelles différentes (proposition 73 de \cite{LCT1}).

\begin{prop}
Si, pour $f\in\homeo_0(\T^2)$, on a $0\in\inte(\rho(\check f))$, alors $f$ possède un itéré tel que les projetés à $\R^2$ des feuilles de tout feuilletage transverse associé à une isotopie maximale sont uniformément bornés.
\end{prop}

\begin{proof}
Puisque $0\in\inte(\rho(\check f))$, et que $\rho(\check f^n) = n\rho(\check f)$, alors quitte à remplacer $f$ par un de ses itérés, l'ensemble $\rho(\check f^n)$ contient les points $(0,1)$, $(0,-1)$, $(1,0)$ et $(-1,0)$. On utilise alors le théorème \ref{ThFranks} : il existe quatre points $z_{(0,1)}$, $z_{(0,-1)}$, $z_{(1,0)}$ et $z_{(-1,0)}$ vérifiant $\check f(z_v) = z_v+v$ (et donc les projetés sur $\T^2$ de ces points sont fixes par $f$).

Soit alors une isotopie maximale $I$ entre $f$ et l'identité, ainsi qu'un feuilletage associé $\F$ sur $\wt\dom(I)$. On note $\check \F$ le projeté de ce feuilletage sur $\R^2$, ainsi que $\check I$ le relevé de $I$ à $\R^2$. Considérons les trajectoires $\check I_\F^\Z(z_v)$ de chacun des points périodiques $z_v$. Par exemple, la trajectoire $\check I_\F^\Z(z_{(1,0)})$ est obtenue comme le saturé sous l'action des translations entières de vecteur $(n,0)$ d'un chemin borné ; c'est donc une droite topologique orientée proprement plongée, et bornée dans la direction verticale. Les mêmes propriétés sont vérifiées par la trajectoire $\check I_\F^\Z(z_{(-1,0)})$. Par conséquent, il existe $N\in\N$ tel que la famille de trajectoires $\check I_\F^\Z(z_{(1,0)}) + (0,2kN)$ et $\check I_\F^\Z(z_{(-1,0)}) + (0,(2k+1)N)$, pour $k$ parcourant $\Z$, est formée de droites deux à deux disjointes. La même propriété est vérifiée par les $\check I_\F^\Z(z_{(0,1)}) + (0,2kN)$ et $\check I_\F^\Z(z_{(0,-1)}) + (0,(2k+1)N)$. 

Sur la figure \ref{FigFeuilBor}, on voit apparaître certaines régions qui sont des attracteurs pour les feuilles (gros points), et pour chacune d'entre elles, si on appelle $R$ l'union des 9 régions qui lui sont sont adjacentes, alors  toute feuille passant par $R$ a son $\alpha$-limite incluse dans $R$. De même, certaines régions sont des répulseurs pour les feuilles (croix), et pour chacune d'entre elles, si on appelle $A$ l'union des 9 régions qui lui sont sont adjacentes, alors  toute feuille passant par $A$ a son $\omega$-limite incluse dans $A$. On en déduit facilement que les feuilles sont bornées dans $\R^2$, en observant que tout point du plan appartient à au moins une des régions $A$ et une des régions $R$.
\end{proof}

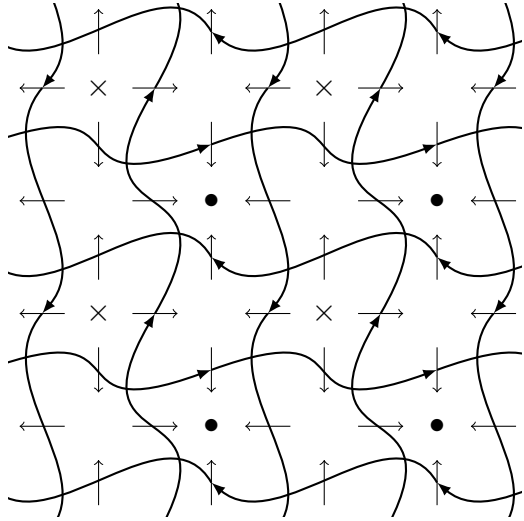
\begin{figure}
\begin{center}
\begin{tikzpicture}[scale=1.5]
\clip(.7,.7) rectangle (5.25,5.25);
\foreach \x in {0,...,3}{
\foreach \y in {0,...,3}{
\draw[thick, ->, >=latex] (2*\x+.5,2*\y) to[in=-160, out=20, looseness=2] (2*\x+2.5,2*\y);
\draw[thick, <-, >=latex] (2*\x+.5,2*\y+1) to[in=-240, out=-40] (2*\x+2.5,2*\y+1);
\draw[thick, ->, >=latex] (2*\x,2*\y+.5) to [in=-120, out=60, looseness=2] (2*\x,2*\y+2.5);
\draw[thick, <-, >=latex] (2*\x+1,2*\y+.5) to [in=-130, out=50, looseness=1] (2*\x+1,2*\y+2.5);

\draw[->] (2*\x-.2,2*\y+.5) -- (2*\x+.2,2*\y+.5);
\draw[->] (2*\x-.2,2*\y+1.5) -- (2*\x+.2,2*\y+1.5);
\draw[<-] (2*\x+.8,2*\y+.5) -- (2*\x+1.2,2*\y+.5);
\draw[<-] (2*\x+.8,2*\y+1.5) -- (2*\x+1.2,2*\y+1.5);
\draw[<-] (2*\x+.5,2*\y-.2) -- (2*\x+.5,2*\y+.2);
\draw[<-] (2*\x+1.5,2*\y-.2) -- (2*\x+1.5,2*\y+.2);
\draw[->] (2*\x+.5,2*\y+.8) -- (2*\x+.5,2*\y+1.2);
\draw[->] (2*\x+1.5,2*\y+.8) -- (2*\x+1.5,2*\y+1.2);

\draw (2*\x+2.5,2*\y+1.5) node {$\bullet$};
\draw (2*\x+1.5,2*\y+2.5) node {$\times$};
}}

\end{tikzpicture}
\caption{Trajectoires transverses associées aux points périodiques $z_v + w$ et directions possibles prises par les feuilles (flèches).}\label{FigFeuilBor}
\end{center}
\end{figure}

Voyons comment cette proposition s'applique pour montrer la positivité de l'entropie.

\begin{proof}[Démonstration du  théorème \ref{TheoRhoInt}]
Considérons les trajectoires $\check I_\F^\Z(z_{(1,0)})$ et $\check I_\F^\Z(z_{(0,1)})$ définies dans la preuve de la proposition précédente, dont on reprend les objets et notations. On voit facilement que certains de leurs relevés s'intersectent $\tilde \F$-transversalement, par le fait que les feuilles du feuilletage sont bornées, et que leurs directions asymptotiques sont différentes. L'application de la proposition fondamentale \ref{PropFond} permet de créer une trajectoire transverse admissible $\gamma^*$ comme sur la figure \ref{FigInterTor} formée de la concaténation d'un bout de la trajectoire $\check I_\F^\Z(z_{(1,0)})$ et d'un bout de la trajectoire $\check I_\F^\Z(z_{(0,1)})$. Puisque les feuilles du feuilletage sont bornées, si ces bouts sont assez longs, alors il existe une translation entière $T$ telle que cette trajectoire possède une intersection $\F$-transverse avec sa translatée $T\gamma^*$. On peut donc appliquer le théorème \ref{ThExistCheval} qui implique l'existence d'un fer à cheval topologique. Cela prouve le premier point du théorème \ref{TheoRhoInt}.

\begin{figure}
\begin{minipage}[c]{0.49\textwidth}
\begin{center}
\begin{tikzpicture}[scale=.6]
\foreach \x in {0,...,2}{
\draw[thick] (2*\x,0) to[in=-190, out=-10, looseness=2] (2*\x+2,0);
\draw[thick] (6,2*\x) to[in=-120, out=60, looseness=2] (6,2*\x+2);

\draw[thick] (2*\x-2.5,-1.5) to[in=-190, out=-10, looseness=2] (2*\x-.5,-1.5);
\draw[thick] (3.5,2*\x-1.5) to[in=-120, out=60, looseness=2] (3.5,2*\x+.5);
}
\draw[->, >=latex] (2,0) to[in=-190, out=-10, looseness=2] (4,0);
\draw[->, >=latex] (6,2) to[in=-120, out=60, looseness=2] (6,4);
\draw[->, >=latex] (-.5,-1.5) to[in=-190, out=-10, looseness=2] (1.5,-1.5);
\draw[->, >=latex] (3.5,.5) to[in=-120, out=60, looseness=2] (3.5,2.5);

\draw (6,2) node[right] {$\gamma^*$};
\draw (3.5,2) node[right] {$T\gamma^*$};
\end{tikzpicture}
\caption{Trajectoire $\gamma^*$ utilisée pour la preuve du théorème \ref{TheoRhoInt}, et sa translatée avec laquelle elle a une intersection $\F$-transverse.}\label{FigInterTor}
\end{center}
\end{minipage}
\hfill
\begin{minipage}[c]{0.49\textwidth}
\begin{center}
\begin{tikzpicture}[scale=.3]
\foreach \x in {0,...,2}{
\draw (2*\x,0) to[in=-190, out=-10, looseness=2] (2*\x+2,0);
\draw[very thick, dashed] (2*\x,0) to[in=-190, out=-10, looseness=2] (2*\x+2,0);
\draw (6,2*\x) to[in=-120, out=60, looseness=2] (6,2*\x+2);
\draw (2*\x+12,1) to[in=-190, out=-10, looseness=2] (2*\x+14,1);
\draw (18,2*\x+1) to[in=-120, out=60, looseness=2] (18,2*\x+3);
}

\draw (6,5) node[right] {$\gamma^*+v_0$};
\draw (18,6) node[right] {$\gamma^*+v_1$};

\draw (1,2) node{$\bullet$} node[below]{$z$} to[in=-180, out=0, looseness=2] (6,2) to[in=-190, out=0, looseness=2] (18,3) to[in=-190, out=-10, looseness=1] (21,2) node{$\bullet$} node[below]{$\check f^n(z)$};

\draw[very thick, dashed] (6,2) to[in=-190, out=0, looseness=2] (18,3) ;
\draw[very thick, dashed](6,0) to[in=-120, out=60, looseness=2] (6,2);
\draw[very thick, dashed] (18,3) to[in=-120, out=60, looseness=2] (18,5) to[in=-120, out=60, looseness=2] (18,7);

\draw[->] (9,0) node[below]{$\gamma(t_0)$} to[bend right] (6.3,1.7);
\draw[->] (21,3.5) node[right]{$\gamma(t_1)$} to[bend right] (18.5,3.3);
\end{tikzpicture}
\caption{Comment créer une trajectoire $\F$-transverse admissible proche de la trajectoire $\gamma$ de $z$ ayant une auto-intersection (en pointillés).}\label{FigInterTor2}
\end{center}
\end{minipage}
\end{figure}

\medskip

La preuve du second point se base sur le même type d'arguments. Supposons qu'il existe $L>0$ \og grand\fg{} ainsi que $n\in\N$ et $z\in\R^2$ tels que $d(\check f^n(z)-z,\, n\rho(\check f)) \ge L$. Considérons la trajectoire $\gamma = \check I_\F^n(z)$. Si cette trajectoire possède une auto-intersection transverse en $\gamma(t_0) + v = \gamma(t_1)$, avec $v\in\Z^2$, $t_0$ proche de 0 et $t_1$ proche de $n$, alors la version rotationnelle du théorème \ref{ThExistCheval} nous donne un fer à cheval rotationnel en temps $m \simeq t_1-t_0$, et en particulier l'existence d'un vecteur $v/m$ dans l'ensemble de rotation. Un petit calcul permet de voir que si $L$ est assez grand et $m$ assez proche de $n$, alors le vecteur $v/m$ n'est pas dans l'ensemble de rotation $\rho(\check f)$, ce qui est une contradiction.

Si une telle intersection transverse n'existe pas, on en crée une artificiellement en se servant du chemin $\gamma^*$ et de la proposition fondamentale (voir la figure \ref{FigInterTor2}) : le chemin $\gamma$ a une intersection transverse avec $\gamma^*+v_0$ en $\gamma(t_0) = \gamma^*(t_0')+v_0$, avec $t_0$ \og proche\fg{} de 0 ; de même $\gamma$ a une intersection transverse avec $\gamma^*+v_1$ en $\gamma(t_1) = \gamma^*(t_1')+v_1$, avec $t_1$ \og proche\fg{} de $n$. Quitte à allonger le début et la fin de $\gamma^* : [0,T]\to \R^2$, on voit qu'il existe $v_2$ \og proche\fg{} de $v_1-v_0$, tel que $\gamma^*|_{[0,t_0']} + v_2$ ait une intersection transverse avec $\gamma^*|_{[t_1',T]}$. Ainsi, le chemin $\big(\gamma^*|_{[0,t_0']} + v_0\big)\gamma|_{[t_0,t_1]}\big( \gamma^*|_{[t_1',T]}+v_1\big)$ est admissible, transverse, et possède une auto-intersection transverse ; comme au-dessus un calcul montre que si $L$ est assez grand, alors on crée un nouveau vecteur de rotation en dehors de l'ensemble de rotation, ce qui est bien évidemment une contradiction.
\end{proof}

La preuve du théorème \ref{ThMZ} sur la forme du bord de l'ensemble de rotation utilise le même type d'arguments. Supposant que 0 est à l'intérieur d'une composante de bord de $\rho(\check f)$ qui est un segment à pente irrationnelle, on commence par voir que là aussi les feuilles du feuilletage transverse sont bornées. L'application du théorème ergodique d'\cite{MR419727} (outil classique de la théorie de la rotation) combinée à la théorie du forçage permettent ensuite de trouver un vecteur de rotation en dehors de l'ensemble de rotation : on aboutit à une contradiction.



\bibliographystyle{smfalpha}
\bibliography{template.bib}

\end{document}